\documentclass[12pt]{amsart}
\usepackage{amsmath, amsfonts, amssymb, latexsym, epsfig, amsthm}
\usepackage[usenames,dvipsnames]{xcolor}
\usepackage{graphicx,pgf,tikz}
\usetikzlibrary{decorations.markings}

\usepackage[all]{xy}
\usepackage{wrapfig}

\usepackage{ytableau}

\newtheorem{theorem}{Theorem}[section]
\newtheorem{proposition}[theorem]{Proposition}
\newtheorem{lemma}[theorem]{Lemma}

\newtheorem{claim}[theorem]{Claim}
\newtheorem*{claim*}{Claim}
\newtheorem{corollary}[theorem]{Corollary}
\newtheorem{Main Conjecture}[theorem]{Main Conjecture}

\theoremstyle{remark}

\newtheorem{example}[theorem]{Example}

\theoremstyle{plain}

\newcommand{\excise}[1]{}

\begin{document}
\pagestyle{plain}
\mbox{}
\title{Partition identities and quiver representations}

\author{Rich\'ard Rim\'anyi}
\address{Department of Mathematics\\
The University of North Carolina at Chapel Hill \\
CB \#3250, Phillips Hall \\
Chapel Hill, NC 27599}
\email{rimanyi@email.unc.edu}

\author{Anna Weigandt}
\author{Alexander Yong}
\address{Department of Mathematics\\
University of Illinois at Urbana-Champaign\\
Urbana, IL 61801}
\email{weigndt2@uiuc.edu, ayong@uiuc.edu}

\date{August 5, 2016}

\maketitle

\begin{abstract}
We present a particular connection between classical partition combinatorics and the theory of quiver representations. Specifically, we give a bijective proof of an analogue of A.~L.~Cauchy's Durfee square identity to multipartitions.  We then use this result to give a new proof of M.~Reineke's identity
in the case of quivers $Q$ of Dynkin type $A$ of arbitrary orientation. Our identity is stated in terms of the lacing diagrams of S.~Abeasis--A.~Del Fra, which parameterize orbits of the representation space of $Q$ for  a fixed dimension vector.
\end{abstract}

\tableofcontents

\section{Introduction}

The main goal of this paper is to establish a specific connection between classical partition combinatorics and the theory of quiver representations. 

\subsection{Lace and (multi)partition combinatorics}
 
   A {\bf lacing diagram} \cite{abeasis1980degenerations} $\mathcal L$ is a graph. The vertices
are arranged in $n$ columns
labeled $1,2,\ldots,n$ (left to right). 
The edges between adjacent columns form a partial matching.
A {\bf strand} is a connected component of~$\mathcal L$.
 
\begin{center}
\begin{tikzpicture}[x=1em,y=1em]

\node[] at (-12,1.5){\begin{tikzpicture}[x=1em,y=1em]
\filldraw [color=black,fill=black,thick](0,1)circle(.1);
\filldraw [color=black,fill=black,thick](0,0)circle(.1);

\filldraw [color=black,fill=black,thick](1,3)circle(.1);
\filldraw [color=black,fill=black,thick](1,2)circle(.1);
\filldraw [color=black,fill=black,thick](1,1)circle(.1);
\filldraw [color=black,fill=black,thick](1,0)circle(.1);

\filldraw [color=black,fill=black,thick](2,2)circle(.1);
\filldraw [color=black,fill=black,thick](2,1)circle(.1);
\filldraw [color=black,fill=black,thick](2,0)circle(.1);

\filldraw [color=black,fill=black,thick](3,1)circle(.1);
\filldraw [color=black,fill=black,thick](3,0)circle(.1);

\draw[color=black,fill=black, thick](0,0)--(1,0);
\draw[color=black,fill=black, thick](0,1)--(1,1);
\draw[color=black,fill=black, thick](1,2)--(2,0);
\draw[color=black,fill=black, thick](2,0)--(3,0);
\draw[color=black,fill=black, thick](1,3)--(2,1);
\end{tikzpicture}
};

\node[] at (-6,1.5){\begin{tikzpicture}[x=1em,y=1em]
\filldraw [color=black,fill=black,thick](0,1)circle(.1);
\filldraw [color=black,fill=black,thick](0,0)circle(.1);

\filldraw [color=black,fill=black,thick](1,3)circle(.1);
\filldraw [color=black,fill=black,thick](1,2)circle(.1);
\filldraw [color=black,fill=black,thick](1,1)circle(.1);
\filldraw [color=black,fill=black,thick](1,0)circle(.1);

\filldraw [color=black,fill=black,thick](2,2)circle(.1);
\filldraw [color=black,fill=black,thick](2,1)circle(.1);
\filldraw [color=black,fill=black,thick](2,0)circle(.1);

\filldraw [color=black,fill=black,thick](3,1)circle(.1);
\filldraw [color=black,fill=black,thick](3,0)circle(.1);

\draw[color=black,fill=black, thick](0,0)--(1,0);
\draw[color=black,fill=black, thick](0,1)--(1,3);
\draw[color=black,fill=black, thick](1,1)--(2,0);
\draw[color=black,fill=black, thick](2,0)--(3,0);
\draw[color=black,fill=black, thick](1,2)--(2,2);
\end{tikzpicture}
};

\node[] at (0,1.5){\begin{tikzpicture}[x=1em,y=1em]
\filldraw [color=black,fill=black,thick](0,1)circle(.1);
\filldraw [color=black,fill=black,thick](0,0)circle(.1);

\filldraw [color=black,fill=black,thick](1,3)circle(.1);
\filldraw [color=black,fill=black,thick](1,2)circle(.1);
\filldraw [color=black,fill=black,thick](1,1)circle(.1);
\filldraw [color=black,fill=black,thick](1,0)circle(.1);

\filldraw [color=black,fill=black,thick](2,2)circle(.1);
\filldraw [color=black,fill=black,thick](2,1)circle(.1);
\filldraw [color=black,fill=black,thick](2,0)circle(.1);

\filldraw [color=black,fill=black,thick](3,1)circle(.1);
\filldraw [color=black,fill=black,thick](3,0)circle(.1);

\draw[color=black,fill=black, thick](0,0)--(1,3);
\draw[color=black,fill=black, thick](0,1)--(1,2);
\draw[color=black,fill=black, thick](1,1)--(2,0);
\draw[color=black,fill=black, thick](2,0)--(3,1);
\draw[color=black,fill=black, thick](1,0)--(2,2);
\end{tikzpicture}
};
\end{tikzpicture}
\end{center}

Two lacing diagrams are {\bf equivalent} if they only
differ by reordering of vertices within columns.  For example, the lacing diagrams pictured above are all equivalent. Let $\eta=[\mathcal L]$ denote the equivalence class of lacing diagrams.

Pick any $\mathcal L\in \eta$ and let ${\mathbf d}(k)$ be the number of vertices in the $k$th column of $\mathcal L$.  Define \[{\bf dim}(\mathcal \eta):=( \mathbf d(1),\ldots, \mathbf  d(n)).\]  Let
\begin{equation}
\label{eqn:sik}
s_{i}^k(\eta)=\#\{\text{strands from column $i$ to column $k-1$}\}, \text{\ and}
\end{equation}
\begin{equation}
\label{eqn:tjk}
t_{j}^k(\eta)=\#\{\text{strands starting at column $j$ using a vertex of column $k$}\}.
\end{equation}

 Fix permutations ${\bf w}=(w^{(1)},\ldots, w^{(n)})$, where $w^{(i)}\in {\mathfrak S}_{i}$ and $w^{(i)}(i)=i$. The partition combinatorics behind Theorem~\ref{thm:main} below suggests the {\bf Durfee statistic}:
\begin{equation}
\label{eqn:durfeeStatistic}
r_{\mathbf w}(\eta)=\sum_{k=2}^n\sum_{1\leq i<j\leq k} s_{w^{(k)}(i)}^k(\eta) t_{w^{(k)}(j)}^k(\eta).
\end{equation} 
We will later attach geometric meaning to $r_{\mathbf w}(\eta)$ (see Theorem~\ref{prop:codim}).

 Let
\[(q)_k=(1-q)(1-q^2)\ldots (1-q^k).\]
L. Euler introduced the following identity of generating series:
\[\frac{1}{(q)_k}=\sum_{r=0}^{\infty} p_{r,k}q^r,\]
where $p_{r,k}$ is the number of {\bf integer partitions} $ \lambda=(\lambda_1\geq \lambda_2\geq \cdots\geq \lambda_{\ell(\lambda)}>0)$ of {\bf size}
$|\lambda|:=\sum {\lambda_i}$ equal to $r$ and parts of size at most $k$.
Therefore it follows that
\[\prod_{k=1}^n\frac{1}{(q)_{\mathbf d(k)}}=\sum_{r=0}^\infty p_{r,\mathbf d} q^r\]
where $p_{r,\mathbf d}$ is the number of sequences of {\bf multipartitions} $(\lambda^{(1)},\ldots,\lambda^{(n)})$ where \[\sum_{i=1}^n|\lambda^{(i)}|=r\] and $\lambda^{(i)}$ has parts of size at most $\mathbf d(i)$.

\begin{theorem}[Quiver Durfee Identity]
\label{thm:main}
\begin{equation}
\label{eqn:maineqn}
\prod_{k=1}^{n} \frac{1}{(q)_{ {\mathbf d}(k)}}=\sum_{\eta}q^{r_{\mathbf w}(\eta)} \prod_{k=1}^n \frac{1}{(q)_{t_k^k(\eta)}}
\prod_{i=1}^{k-1}{t_i^k(\eta)+s_i^k(\eta)\brack s_i^k(\eta)}_q
, 
\end{equation}
where the sum is taken over $\eta$ such that ${\bf dim}(\eta)= (\mathbf d(1),\ldots,\mathbf d(n))$.
\end{theorem}

Here
\[{k\brack j}_q=\frac{[k]_q!}{[j]_q![k-j]_q!}=\frac{(q)_k}{(q)_j(q)_{k-j}}\] 
is the {\bf Gaussian binomial coefficient}, where 
$[i]_q:=1+q+q^2+\cdots+q^{i-1}$.   In fact, ${k\brack j}_q$ is the generating
series for partitions whose associated Ferrers shape is contained
in a $j\times (k-j)$ rectangle. That is
\[{k\brack j}_q=\sum_{\lambda\subseteq j\times (k-j)}
q^{|\lambda|}.\]

\begin{example}[Relationship to classical Durfee square identity]
\label{exa:classical}
Let $n=2$ and set $\mathbf d(1)=\mathbf d(2)=k$.  Then $w^{(1)}=1$ and $w^{(2)}=1 2$ (throughout we will express permutations in one line notation) by the assumption $w^{(k)}(k)=k$.   Equivalence classes of lacing diagrams are determined by the number of strands which start and end at the first vertex.  If there are $j$ such strands, then there are $k-j$ strands connecting the first and second vertex.  Then there must be exactly $k-(k-j)=j$ strands starting and ending at the second vertex.

\begin{center}
\begin{tikzpicture}[x=1em,y=1em]
\filldraw [color=black,fill=black,thick](0,3)circle(.1);
\filldraw [color=black,fill=black,thick](0,2)circle(.1);
\filldraw [color=black,fill=black,thick](0,1)circle(.1);
\filldraw [color=black,fill=black,thick](0,0)circle(.1);

\filldraw [color=black,fill=black,thick](1,3)circle(.1);
\filldraw [color=black,fill=black,thick](1,2)circle(.1);
\filldraw [color=black,fill=black,thick](1,1)circle(.1);
\filldraw [color=black,fill=black,thick](1,0)circle(.1);

\draw[color=black,fill=black, thick](0,0)--(1,0);
\draw[color=black,fill=black, thick](0,1)--(1,1);
\draw[color=black,fill=black, thick](0,2)--(1,2);

\node[] at (3.5,1){$\bigg\} \quad  k-j$};
\node[] at (2.5,3){$\} \quad j$};
\end{tikzpicture}
\end{center}
  So if $\eta$ has $j$ strands of type $[1,1]$, then \[s_1^2(\eta)=j, \ \ t_1^1(\eta)=j, \ \ t_1^2(\eta)=k-j, \text{\ \ and $t_2^2(\eta)=j$.}\] 
Thus 
\[r_{\mathbf w}(\eta)=s_1^2(\eta)t_2^2(\eta)=j^2.\]  
Hence (\ref{eqn:maineqn}) states 
\[\frac{1}{(q)_k}\frac{1}{(q)_k}=\sum_{j=0}^k q^{j^2} \frac{1}{(q)_k} \frac{1}{(q)_{j}}{(k-j)+j\brack j}_q.\]
Multiplying both sides by $(q)_{k}$ gives
the ``Durfee square identity'' due to A-L. Cauchy:
\begin{equation}
\label{eqn:Gauss}
\frac{1}{(q)_k}=\sum_{j= 0}^k q^{j^2} 
{k\brack j}_q \frac{1}{(q)_j}.
\end{equation}

The {\bf Durfee square} $D(\lambda)$ of $\lambda$
is the largest $j\times j$ square that fits inside $\lambda$. Let ${\mathcal P}_k$ be the set of partitions of width at most $k$. By decomposing
$\lambda$ using $D(\lambda)$ one obtains a
bijection
${\mathcal P}_k\xrightarrow{\sim} \bigcup_{j\geq 0} {\mathcal D}\times
{\mathcal A}_j\times {\mathcal P}_j$
where ${\mathcal D}$ is the singleton set consisting of the $j\times j$ square
and
${\mathcal A}_j$ is the set of partitions contained in a $j\times (k-j)$ rectangle.  This gives a textbook bijective proof of (\ref{eqn:Gauss}).\qed
\end{example}

There has been earlier work generalizing the Durfee square
identity to multipartitions. In particular, we point the reader to
the definition of  \emph{Durfee dissections} of A. Schilling \cite{schilling1998supernomial}, which has some similarities in shape to
the identity of Theorem~\ref{thm:main}.    Here, each \emph{Durfee rectangle} has at least as many columns as rows, which differs from our definition.  We also note the resemblance to the \emph{Durfee systems} of P. Bouwknegt \cite{bouwknegt2002multipartitions}.     Also see the references to
\emph{loc.~cit.} for other work on generalized Durfee square
identities.   One main point of difference is that these identities do not
concern lacing diagrams. 

\begin{example}
\label{example:n3identity}
Let $n=3$ and $\mathbf d=(1,2,1)$ and $\mathbf w=(1,12,123)$.  Then \[r_{\mathbf w}=(s_1^2t_2^2)+(s_1^3t_2^3+s_1^3t_3^3+s_2^3t_3^3)\]
and
\[ \prod_{k=1}^3 \frac{1}{(q)_{t_k^k}}
\prod_{i=1}^{k-1}{t_i^k+s_i^k\brack s_i^k}_q=\left(\frac{1}{(q)_{t_1^1}}\right)\left(\frac{1}{(q)_{t_2^2}} {t_1^2+s_1^2\brack s_1^2}_q \right) \left (\frac{1}{(q)_{t_3^3}}{t_1^3+s_1^3\brack s_1^3}_q {t_2^3+s_2^3\brack s_2^3}_q \right).\]
The table below gives the equivalence classes for $\mathbf d=(1,2,1)$ and their corresponding terms on the right hand side of (\ref{eqn:maineqn}).
\begin{center}
\begin{tabular}{|c|c|c|c|} \hline
$[\mathcal L]$& $(s_j^k)$ & $(t_j^k)$&$q^{r_w}\left(\frac{1}{(q)_{t_1^1}}\right)\left(\frac{1}{(q)_{t_2^2}} {t_1^2+s_1^2\brack s_1^2}_q \right) \left (\frac{1}{(q)_{t_3^3}}{t_1^3+s_1^3\brack s_1^3}_q {t_2^3+s_2^3\brack s_2^3}_q \right)$ \\\hline
$\left[\begin{tikzpicture}[x=1em,y=1em]
\filldraw [color=black,fill=black,thick](0,0)circle(.1);
\filldraw [color=black,fill=black,thick](1,1)circle(.1);
\filldraw [color=black,fill=black,thick](1,0)circle(.1);
\filldraw [color=black,fill=black,thick](2,0)circle(.1);
\end{tikzpicture}\right]$
&
$\begin{array}{cc|c}
2 & 1& j/k\\ \hline
&1&2\\
2&0&3\\
\end{array}$
&
$\begin{array}{ccc|c}
3 & 2 & 1& j/k\\ \hline
&&1&1\\
&2&0&2\\
1&0&0&3\\
\end{array}$
&
$q^4\left(\frac{1}{(q)_{1}}\right)\left(\frac{1}{(q)_{2}}\right) \left (\frac{1}{(q)_{1}}  \right)=\frac{q^4}{(1-q)^3(1-q^2)}$
\\ \hline
$\left[\begin{tikzpicture}[x=1em,y=1em]
\filldraw [color=black,fill=black,thick](0,0)circle(.1);
\filldraw [color=black,fill=black,thick](1,1)circle(.1);
\filldraw [color=black,fill=black,thick](1,0)circle(.1);
\filldraw [color=black,fill=black,thick](2,0)circle(.1);

\draw[color=black,fill=black, thick](0,0)--(1,0);
\end{tikzpicture}\right]$
&
$\begin{array}{cc|c}
 2 & 1& j/k\\ \hline
&0&2\\
1&1&3\\
\end{array}$
&
$\begin{array}{ccc|c}
3 & 2 & 1& j/k\\ \hline
&&1&1\\
&1&1&2\\
1&0&0&3\\
\end{array}$
&
$q^2\left(\frac{1}{(q)_{1}}\right)\left(\frac{1}{(q)_{1}}  \right) \left (\frac{1}{(q)_{1}} \right)=\frac{q^2}{(1-q)^3}$
\\ \hline

$\left[\begin{tikzpicture}[x=1em,y=1em]
\filldraw [color=black,fill=black,thick](0,0)circle(.1);
\filldraw [color=black,fill=black,thick](1,1)circle(.1);
\filldraw [color=black,fill=black,thick](1,0)circle(.1);
\filldraw [color=black,fill=black,thick](2,0)circle(.1);

\draw[color=black,fill=black, thick](1,0)--(2,0);
\end{tikzpicture}\right]$
&
$\begin{array}{cc|c}
2 & 1& j/k\\ \hline
&1&2\\
1&0&3\\
\end{array}$
&
$\begin{array}{ccc|c}
3 & 2 & 1& j/k\\ \hline
&&1&1\\
&2&0&2\\
0&1&0&3\\
\end{array}$
&
$q^2\left(\frac{1}{(q)_{1}}\right)\left(\frac{1}{(q)_{2}} \right) \left ( {2\brack 1}_q \right)=\frac{q^2}{(1-q)^3}$
\\ \hline
$\left[\begin{tikzpicture}[x=1em,y=1em]
\filldraw [color=black,fill=black,thick](0,0)circle(.1);
\filldraw [color=black,fill=black,thick](1,1)circle(.1);
\filldraw [color=black,fill=black,thick](1,0)circle(.1);
\filldraw [color=black,fill=black,thick](2,0)circle(.1);

\draw[color=black,fill=black, thick](0,0)--(1,0);
\draw[color=black,fill=black, thick](1,1)--(2,0);
\end{tikzpicture}\right]$
&
$\begin{array}{cc|c}
 2 & 1& j/k\\ \hline
&0&2\\
0&1&3\\
\end{array}$
&
$\begin{array}{ccc|c}
3 & 2 & 1& j/k\\ \hline
&&1&1\\
&1&1&2\\
0&1&0&3\\
\end{array}$
&
$q\left(\frac{1}{(q)_{1}}\right)\left(\frac{1}{(q)_{1}} \right) =\frac{q}{(1-q)^2}$
\\ \hline
$\left[\begin{tikzpicture}[x=1em,y=1em]
\filldraw [color=black,fill=black,thick](0,0)circle(.1);
\filldraw [color=black,fill=black,thick](1,1)circle(.1);
\filldraw [color=black,fill=black,thick](1,0)circle(.1);
\filldraw [color=black,fill=black,thick](2,0)circle(.1);

\draw[color=black,fill=black, thick](0,0)--(2,0);
\end{tikzpicture}\right]$
&
$\begin{array}{cc|c}
 2 & 1& j/k\\ \hline
&0&2\\
1&0&3\\
\end{array}$
&
$\begin{array}{ccc|c}
3 & 2 & 1& j/k\\ \hline
&&1&1\\
&1&1&2\\
0&0&1&3\\
\end{array}$
&
$\left(\frac{1}{(q)_{1}}\right)\left(\frac{1}{(q)_{1}}  \right) =\frac{1}{(1-q)^2}$
\\ \hline
\end{tabular}
\end{center}
We then verify,
\begin{align*}
{\rm RHS}&=\frac{q^4}{(1-q)^3(1-q^2)}+\frac{q^2}{(1-q)^3}+\frac{q^2}{(1-q)^3}+\frac{q}{(1-q)^2}+\frac{1}{(1-q)^2}\\
&=\frac{1}{(1-q)^3(1-q^2)}(q^4+q^2(1-q^2)+q^2(1-q^2)+q(1-q)(1-q^2)+(1-q)(1-q^2))\\
&=\frac{1}{(1-q)^3(1-q^2)}\\
&=\frac{1}{(q)_1(q)_2(q)_1}\\
&={\rm LHS}.
\end{align*}

Notice that (\ref{eqn:Gauss}) says
\[\frac{1}{(q)_1}=1+\frac{q}{(q)_1}\]
and 
\[\frac{1}{(q)_2}=1+\frac{q}{(q)_1}{2\brack 1}_q+\frac{q^4}{(q)_2}\]
Thus
\begin{align*}
\frac{1}{(q)_1}\frac{1}{(q)_2}\frac{1}{(q)_1}&=\left(\frac{1}{(q)_1}\right)\left(1+\frac{q}{(q)_1}\right)\left(1+\frac{q}{(q)_1}{2\brack 1}_q+\frac{q^4}{(q)_2}\right)\\
&=\frac{1}{1-q}+\frac{q}{(1-q)^2}+\frac{q(1+q)}
{(1-q)^2}+\frac{q^2(1+q)}{(1-q)^3}+\frac{q^4}{(1-q)^2(1-q^2)}+\frac{q^5}{(1-q)^3(1-q^2)}
\end{align*}
Theorem~\ref{thm:main} does not appear to be an \emph{a priori} consequence of (\ref{eqn:Gauss}).  Instead, we will give a \emph{bijective} proof of Theorem~\ref{thm:main}
in the spirit of the one given for (\ref{eqn:Gauss})
in Example~\ref{exa:classical}.
\qed
\end{example}

A strand is of {\bf type} $[i,j]$ if it starts in column $i$ and ends in column $j$.  The number of strands of type $[i,j]$ is invariant on $[\mathcal L]$.  Therefore we let 
\begin{equation}
\label{eqn:mij}
m_{[i,j]}(\eta)= \#\{\text{ strands of type $[i,j]$ in any $\mathcal L$ of $\eta=[\mathcal L]$}\}.
\end{equation}

\begin{corollary}
\label{cor:cancel}
\begin{equation}
\label{eqn:cancel}
\prod_{i=1}^ n \frac{1}{(q)_{ {\mathbf d}(i)}}=\sum_{\eta}q^{r_{\mathbf w}(\eta)}\prod_{1\leq i\leq j\leq n} \frac{1}{(q)_{m_{[i,j]}(\eta)}}.
\end{equation}
\end{corollary}
\begin{proof}
From the definitions,
\begin{equation}
t_i^k(\eta)+s_i^k(\eta)=t_i^{k-1}(\eta).
\end{equation} Furthermore, 
\[s_i^k(\eta)=m_{[i,k-1]}(\eta) \text{\  and $t_i^n(\eta)=m_{[i,n]}(\eta)$.}\]  
Thus,
\begin{align*}
\prod_{k=1}^n \frac{1}{(q)_{t_k^k(\eta)}} \prod_{i=1}^{k-1}{t_i^k(\eta)+s_i^k(\eta)\brack s_i^k(\eta)}_q
&=\prod_{k=1}^n \frac{1}{(q)_{t_k^k(\eta)}} \prod_{i=1}^{k-1}\frac{(q)_{t_i^k(\eta)+s_i^k(\eta)}}{ (q)_{t_i^k(\eta)}(q)_{s_i^k(\eta)}}\\
&=\prod_{k=1}^n \frac{1}{(q)_{t_k^k(\eta)}} \prod_{i=1}^{k-1}\frac{(q)_{t_i^{k-1}(\eta)}}{ (q)_{t_i^k(\eta)}(q)_{s_i^k(\eta)}}\\
&=\left(\prod_{k=1}^n\frac{1}{(q)_{t_k^k(\eta)}}\prod_{i=1}^{k-1}\frac{(q)_{t_i^{k-1}(\eta)}}{(q)_{t_i^k}(\eta)}\right) \left(\prod_{k=1}^n\prod_{i=1}^{k-1}\frac{1}{(q)_{s_i^k(\eta)}}\right)\\
&=\left(\prod_{k=1}^n\prod_{i=1}^k\frac{1}{(q)_{t_i^k(\eta)}}\right)\left(\prod_{k=2}^n\prod_{i=1}^{k-1}(q)_{t_i^{k-1}(\eta)}\right) \left(\prod_{k=1}^n\prod_{i=1}^{k-1}\frac{1}{(q)_{s_i^k(\eta)}}\right)\\
&=\left(\prod_{k=1}^n\prod_{i=1}^k\frac{1}{(q)_{t_i^k(\eta)}}\right)\left(\prod_{k=1}^{n-1}\prod_{i=1}^{k}(q)_{t_i^{k}(\eta)}\right) \left(\prod_{k=1}^n\prod_{i=1}^{k-1}\frac{1}{(q)_{s_i^k(\eta)}}\right)\\
&=\left(\prod_{i=1}^n\frac{1}{(q)_{t_i^n(\eta)}}\right)\left(\prod_{k=1}^n\prod_{i=1}^{k-1}\frac{1}{(q)_{s_i^k(\eta)}}\right)\\
&=\left(\prod_{i=1}^n\frac{1}{(q)_{m_{[i,n]}(\eta)}}\right)\left(\prod_{k=1}^n\prod_{i=1}^{k-1}\frac{1}{(q)_{m_{[i,k-1]}(\eta)}}\right)\\
&=\prod_{1\leq i\leq j\leq n}\frac{1}{(q)_{m_{[i,j]}(\eta)}}. \qedhere
\end{align*}
\end{proof}

\subsection{Quiver Representations}
M.~Reineke (cf.~\cite[(10)]{rimanyi2013cohomological}) proved an identity \emph{very} close to (\ref{eqn:cancel})
 that is the motivation of this work. His identity
is phrased in terms of quiver representations; we briefly recall the background essentials.   One source concerning quiver representations is \cite{brion2008representations}.

Let $Q$ be a {\bf quiver}, a directed graph with vertex set $Q_0$ and arrows $Q_1$. For $a\in Q_1$ let $h(a)$ be the head of the arrow and $t(a)$ its tail.  Throughout we will work over $\mathbb C$.
A {\bf representation} ${\sf V}$ of $Q$ assigns a vector space $V_x$ to each $x\in Q_0$ as well as a 
linear transformation $V_a:V_{t(a)}\rightarrow V_{h(a)}$ for each arrow $a\in Q_1$.   
Each representation ${\sf V}$ of $Q$ has an associated {\bf dimension vector} 
\[{\mathbf d}:Q_0\rightarrow \mathbb Z_{\geq 0},
\text{\ where $ \mathbf {\mathbf d}(x)={\rm dim} V_x$.}\] 

A ${\bf morphism}$ ${\sf T}:{\sf V}\rightarrow {\sf W}$ is a collection of linear maps $(T_x:V_x\rightarrow W_x)_{x\in Q_0}$ such that 
  \[T_{h(a)}V_a=W_a T_{t(a)}
\text{\ for every arrow $a\in Q_1$.}\]
    Write ${\rm Hom}(\sf V,\sf W)$ for the space of morphisms from $\sf V$ to $\sf W$. 
    Given representations ${\sf V}$ and ${\sf W}$, we may form the {\bf direct sum} ${\sf V}\oplus {\sf W}$ by pointwise taking direct sums of vector spaces and morphisms.  If ${\sf V}\cong {\sf V'}\oplus {\sf V''}$ implies ${\sf V'}$ or ${\sf V''}$ is trivial, then ${\sf V}$ is {\bf indecomposable}.  
 If ${\sf V}$ is a finite dimensional representation of $Q$ then the Krull-Schmidt decomposition is 
 \begin{equation}
 \label{eqn:Krull}
   {\sf V}\cong \bigoplus_{i=1}^m {\sf V_i}^{\oplus m_i},
   \end{equation}
     where the ${\sf V_i}$ are pairwise non-isomorphic indecomposable representations.  This decomposition and the multiplicities $m_i$ are unique up to reordering.

Let ${\sf Mat}(m,n)$ be the space of $m \times n$ matrices.
The {\bf representation space} is 
\[{\sf Rep}_Q( \mathbf d):=\bigoplus_{a\in Q_1}{\sf Mat}({\mathbf d}(h(a)),{\mathbf d}(t(a)).\]  ${\sf Rep}_Q(\mathbf d)$ is isomorphic to affine space $\mathbb A^N$ where $N=\sum_{a\in Q_1}\mathbf d(h(a))\mathbf d(t(a))$.  Points of ${\sf Rep}_Q(\mathbf d)$ parameterize $\mathbf d$ dimensional representations of $Q$.
Let 
\[{\sf GL}_Q( \mathbf d):=\prod_{x\in Q_0} {\sf GL}({ \mathbf d}(x)).\]  ${\sf GL}_Q(\mathbf d)$ acts on ${\sf Rep}_Q(\mathbf d)$ by base change.  Orbits of this action are in bijection with isomorphism classes of $\mathbf d$ dimensional representations.

For the remainder of the paper, assume $Q$ is a type $A_n$ quiver, i.e. the underlying graph of $Q$ is a  path with $n$ vertices.  Then ${\sf GL}_Q(\mathbf d)$ acts on ${\sf Rep}_Q(\mathbf d)$ with finitely many orbits.  In particular, these orbits are indexed by equivalence classes of $\mathbf d$-dimensional lacing diagrams, as follows.  

Identify the vertices of $Q$ with the numbers $1,\ldots, n$ from left to right. Let \[\Phi^+=\{I=[i,j]:1\leq i\leq j\leq n\}\] be the set of intervals in $Q$.   Label the arrows of $Q$ from left to right $a_1$ through $a_{n-1}$.  In this case,
P.~Gabriel's theorem states that isomorphism classes of indecomposables biject with elements of $\Phi^+$ in the following way.  Define ${\sf V}_I$ with vector spaces
\[({\sf V}_I)_k=\begin{cases}\mathbb C &\text{if $k\in I$}\\ 0 &\text{otherwise}\end{cases}\] and morphisms
\[({\sf V}_I)_{a}=\begin{cases}\rm{id}:\mathbb C\rightarrow \mathbb C &\text{if $h(a),t(a)\in I$}\\ 0 &\text{otherwise.}\end{cases}\]
Then by (\ref{eqn:Krull}),
\[{\sf V}\cong \bigoplus_{I\in \Phi^+}{ \sf V}_{I}^{\oplus m_{I}},\]
where $m_{[i,j]}$ is the multiplicity of ${\sf V}_{I}$ in $\sf V$.  We record this data in a lacing diagram $\mathcal L$ which has $m_{[i,j]}$ strands starting in column $i$ and ending in column $j$.

Let $\mathbf d={\bf dim}(\eta)$.  Write 
\[\mathcal O_\eta:={\sf GL}_Q(\mathbf d)\cdot V_\eta\subset {\sf Rep}_Q(\mathbf d)\] 
where \[V_\eta:= \bigoplus_{I\in \Phi^+}{ \sf V}_{I}^{\oplus m_{I}}.\]  Write ${\rm codim}_{\mathbb C}(\eta)$ for the (complex) codimension of $\mathcal O_\eta$ in $ {\sf Rep}_Q(\mathbf d)$.

\begin{corollary}[M.~Reineke's identity for type $A_n$ quivers]
\label{eqn:Reineke}
 For a fixed dimension vector $\mathbf d$:
\[ \prod_{i=1}^n\frac{1}{(q)_{{\mathbf d}(i)}}=\sum_\eta q^{{\rm codim}_\mathbb C\eta}\prod_{I\in \Phi^+} \frac{1}{(q)_{m_{I}(\eta)}},\]
where the sum is taken over $\eta$ so that ${\bf dim}(\eta)=\mathbf d$.
\end{corollary} 

M.~Reineke's identity holds more generally for
all $ADE$ Dynkin types. It should be possible to
treat the other cases in a similar manner, although we do not do so here.

Reineke's identities may be naturally phrased as identities among quantum dilogarithm power series in a non-commutative ring. In this language the identities are closely related to cluster algebras (see e.g., work of V.~V.~Fock--A.~B.~Goncharov \cite{fock} and references therein), wall crossing phenomena (see e.g.,
the paper \cite{davison} of B.~Davison--S.~Meinhardt as well as the references therein), and Donaldson-Thomas invariants and Cohomological Hall Algebras (see, e.g., the work of M.~Kontsevich--Y.~Soibelman \cite{kontsevich}). This paper is intended to be an initial step towards understanding the rich combinatorics encoded by advanced dilogarithm identities, such as B.~Keller's identities \cite{keller}.

We now explain our proof of Corollary~\ref{eqn:Reineke}
as a special case of Corollary~\ref{cor:cancel} where ${\mathbf w}$ is
determined by $Q$. 
We define permutations $w_Q^{(i)}\in \mathfrak S_i$ as follows.   Let $w_Q^{(1)}=1$ and $w_Q^{(2)}=1 2$.  For $i\geq 3$ let $\iota$ be the natural inclusion from $\mathfrak S_{i-1}$ to $\mathfrak S_{i}$ and let $w_0^{(i-1)}$ denote the longest permutation in $\mathfrak S_{i-1}$.  
Then we set
\[w_Q^{(i)}= \begin{cases}
\iota(w_Q^{(i-1)}) \text{\ \ if $a_{i-2}$ and $a_{i-1}$ point in the same direction }\\ 
\iota( w_Q^{(i-1)}w_0^{(i-1)}) \text{\ \  if $a_{i-2}$ and $a_{i-1}$ point in opposite directions.} 

\end{cases}\]
Write ${\mathbf w_Q}=(w_Q^{(1)},\ldots, w_Q^{(n)})$.

\begin{example}
Let $Q$ be the quiver pictured below.
\begin{center}
\begin{tikzpicture}[x=.75cm,y=.75cm]
\filldraw [color=black,fill=black,thick](0,0)circle(.1);
\filldraw [color=black,fill=black,thick](1,0)circle(.1);
\filldraw [color=black,fill=black,thick](2,0)circle(.1);
\filldraw [color=black,fill=black,thick](3,0)circle(.1);
\filldraw [color=black,fill=black,thick](4,0)circle(.1);
\filldraw [color=black,fill=black,thick](5,0)circle(.1);
\draw[ decoration={markings, mark=at position 0.6 with {\arrow[scale=2]{>}}}, postaction={decorate}](4,0) -- (5,0);
\draw[ decoration={markings, mark=at position 0.6 with {\arrow[scale=2]{<}}}, postaction={decorate}](3,0) -- (4,0);
\draw[ decoration={markings, mark=at position 0.6 with {\arrow[scale=2]{<}}}, postaction={decorate}](2,0) -- (3,0);
\draw[ decoration={markings, mark=at position 0.6 with {\arrow[scale=2]{>}}}, postaction={decorate}](1,0) -- (2,0);
\draw[ decoration={markings, mark=at position 0.6 with {\arrow[scale=2]{>}}}, postaction={decorate}](0,0) -- (1,0);
\node[] at (0,-.5){$1$};
\node[] at (1,-.5){$2$};
\node[] at (2,-.5){$3$};
\node[] at (3,-.5){$4$};
\node[] at (4,-.5){$5$};
\node[] at (5,-.5){$6$};
\node[] at (.5,.5){$a_1$};
\node[] at (1.5,.5){$a_2$};
\node[] at (2.5,.5){$a_3$};
\node[] at (3.5,.5){$a_4$};
\node[] at (4.5,.5){$a_5$};
\end{tikzpicture}
\end{center}

Then $Q$ has associated permutations 
$\mathbf w_Q=(1,12,123,3214,32145,541236)$.\qed
\end{example}

With this, it remains to show that the
Durfee statistic computes codimension:
\begin{theorem}
\label{prop:codim}
\[r_{\mathbf w_Q}(\eta)={\rm codim}_\mathbb C(\mathcal O_\eta). \]
\end{theorem}

We arrive at Theorem~\ref{prop:codim} by connecting $r_{\mathbf w_Q}(\eta)$ to an earlier positive combinatorial formula for ${\rm codim}_\mathbb C(\mathcal O_\eta)$.

\section{Proof of Theorem~\ref{thm:main}}

Recall the left hand side of (\ref{eqn:maineqn}) is the generating series
for an $n$-tuple of partitions, i.e., 
\[S=\{\boldsymbol \lambda =(\lambda^{(k)})_{1\leq k\leq n}:\lambda^{(k)} \text{ is a partition having parts of size at most }  {\mathbf d}(k)\}\]
with respect to the weight:
 \[{\tt wt}_S(\boldsymbol\lambda)=\sum_{k=1}^n |\lambda^{(k)}|.\]
Consider the one element set
  \[R(\eta)=\{\boldsymbol \mu=(\mu_{i,j}^k):\mu_{i,j}^k \text{ is a } s_{w^{(k)}(i)}^k(\eta)\times t_{w^{(k)}(j)}^k(\eta) \text{ rectangle}, 1\leq i <j\leq k\leq n\},\] 
consisting of a list of rectangles depending on $i$, $j$, and $k$. Then $r_{\mathbf w}(\eta)$ is the total number of boxes in this list of rectangles.

For $i <k $, let $P_{i}^k(\eta)$ be the set of partitions which fit inside of an $s_i^k(\eta)\times t_{i}^k(\eta)$ box.  Also let $P_k^k(\eta)$ be the set of partitions which have parts of size at most $t_k^k(\eta)$.  Let \[P(\eta)=\{\boldsymbol \nu=(\nu_i^k):\nu_i^k\in P_{w^{(k)}(i)}^k(\eta), 1\leq i\leq k \leq n\}.\]  Set 
\[T({\eta})=R(\eta)\times P(\eta).\]  
Finally, we let 
\[T=\bigcup_{\eta} T(\eta),\] 
with the union taken over all lace equivalence classes $\eta$ of dimension $ \mathbf d$. 

The right hand side of (\ref{eqn:maineqn}) is the generating series 
for $T$, with respect to the weight
that assigns
$(\boldsymbol\mu,\boldsymbol\nu)\in T$ to 
\[{\tt wt}_T(\boldsymbol \mu,\boldsymbol\nu)=\sum_{1\leq i<j<k\leq n} |\mu_{i,j}^k|+\sum_{1\leq i\leq k\leq n} |\nu_i^k|.\]

  Define a map $\Psi:T\to S$ by ``gluing'' the partitions of $T$
as indicated in Figure~\ref{figure:firstglue}, for $1\leq k\leq n$.

Thus, Theorem~\ref{thm:main} follows from:
\begin{theorem}
\label{thm:bij}
$\Psi:T\to S$
is a weight-preserving bijection, i.e.,
${\tt wt}_T(\boldsymbol\mu,\boldsymbol\nu)
={\tt wt}_S(\Psi(\boldsymbol\mu,\boldsymbol\nu))$.
\end{theorem}

\begin{figure}
\begin{picture}(200,280)
\put(20,0){\includegraphics[scale=2]{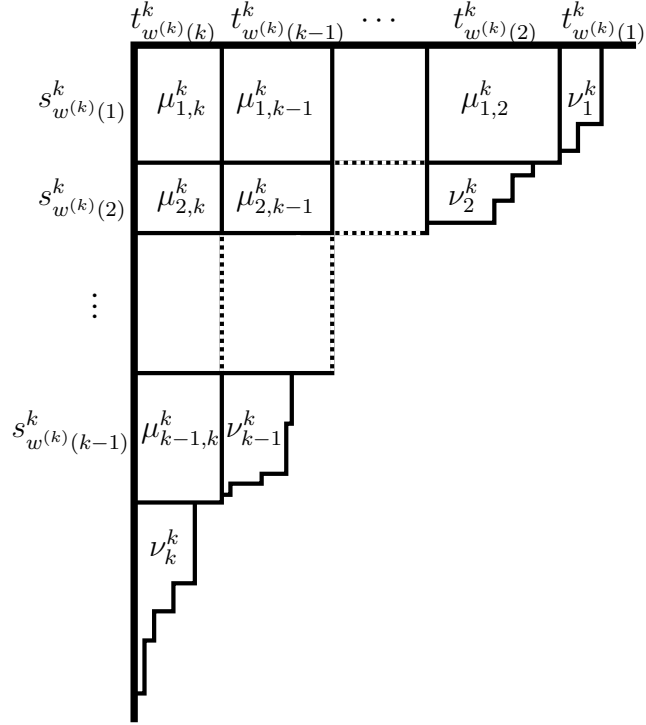}}
\put(20,265){$t_{w^{(k)}(k)}^k$}
\put(58,265){$t_{w^{(k)}(k-1)}^k$}
\put(142,265){$t_{w^{(k)}(2)}^k$}
\put(183,265){$t_{w^{(k)}(1)}^k$}
\put(-15,235){$s_{w^{(k)}(1)}^k$}
\put(-15,198){$s_{w^{(k)}(2)}^k$}
\put(-25,110){$s_{w^{(k)}(k-1)}^k$}
\put(107,265){$\cdots$}
\put(5,155){$\vdots$}

\put(30,235){$\mu_{1,k}^k$}
\put(60,235){$\mu_{1,k-1}^k$}
\put(145,235){$\mu_{1,2}^k$}
\put(30,198){$\mu_{2,k}^k$}
\put(60,198){$\mu_{2,k-1}^k$}
\put(24,110){$\mu_{k-1,k}^k$}

\put(56,110){$\nu_{k-1}^k$}
\put(185,235){$\nu_{1}^k$}
\put(140,198){$\nu_{2}^k$}
\put(27,65){$\nu_{k}^k$}
\end{picture}
\caption{Description of the $k$-th component of the map $\Psi:T\to S$}

\label{figure:firstglue}
\end{figure}

\begin{proof}

{\sf $\Psi$ is well-defined}:
This follows immediately from that fact that if
${\bf dim}(\eta)= \mathbf d$ then \[t_1^k(\eta)+\ldots + t_k^k(\eta)= {\mathbf d}(k).\]  

\noindent
{\sf $\Psi$ is weight-preserving}:
That ${\tt wt}_T(\boldsymbol \mu,\boldsymbol \nu))={\tt wt}_S(\Psi(\boldsymbol \mu,\boldsymbol \nu))$ is clear since $\Psi$ preserves the total number of
boxes.

\noindent
{\sf Definition of  $\Phi:S\rightarrow T$:}
Given a partition $\lambda$, and $i\in \mathbb Z$, 
the {\bf Durfee rectangle} $D(\lambda,i)$ is the rectangle with top left corner positioned at $(0,0)$ 
and bottom right corner where the line $x+y=i$ intersects the (infinite) boundary line of the partition. 
Equivalently, this is the largest $s\times(s+i)$ rectangle which fits in $\lambda$, justified against the top left corner.
(By convention, we define $0$-width and $0$-height rectangles as fitting in $\lambda$.)

\begin{example}
\label{example:durfee}
Let $ \lambda=(3,3,2,2,1)$.  Pictured below are the Durfee rectangles $D(\lambda,i)$ for $i=-1,0,4$.  
\setlength{\unitlength}{1em}
\ytableausetup{boxsize=1em}
\begin{center}
\begin{tikzpicture}[x=1em,y=1em]
\node[ below right] at (0,7){
\begin{ytableau}
*(gray) & *(gray) &    \\
*(gray)  & *(gray) &  \\
 *(gray)  &*(gray)   \\
  &  \\
   \\
\end{ytableau}};
\node[ below right]  at (0,8.25){$D(\lambda,-1)$};
\draw[very thick](-1.65,6.65)--(5.35,6.65);
\draw[very thick](.35,6.65)--(.35,-.35);
\draw[very thick](-.65,6.65)--(2.45,3.5);
\draw[line width=.2em](.35,6.65)rectangle (2.45,3.5);
\filldraw [color=black,fill=black,thick](2.5,3.5)circle(.2);
\filldraw [color=black,fill=black,thick](-.6,6.6)circle(.2);

\node[ below right] at (8,7){
\begin{ytableau}
*(gray) & *(gray) &    \\
*(gray)  & *(gray) &  \\
   &   \\
  &  \\
   \\
\end{ytableau}};

\node[ below right]  at (8,8.25){$D(\lambda,0)$};
\draw[very thick](8.30,6.65)--(13.35,6.65);
\draw[very thick](8.35,6.65)--(8.35,-.35);
\draw[very thick](8.35,6.65)--(10.45,4.5);
\draw[line width=.2em](8.35,6.65)rectangle (10.45,4.5);
\filldraw [color=black,fill=black,thick](10.5,4.5)circle(.2);
\filldraw [color=black,fill=black,thick](8.35,6.6)circle(.2);

\node[ below right] at (16,7){
\begin{ytableau}
\ &  &    \\
& &  \\
   &   \\
  &  \\
   \\
\end{ytableau}};

\node[ below right]  at (16,8.25){$D(\lambda,4)$};
\draw[very thick](16.30,6.65)--(21.35,6.65);
\draw[very thick](16.35,6.65)--(16.35,-.35);
\draw[line width=.2em](16.35,6.6)--(20.45,6.6);
\filldraw [color=black,fill=black,thick](20.5,6.6)circle(.2);

\end{tikzpicture}
\end{center}
Notice that $D(\lambda,4)=0\times 4$ rectangle.  The line $x+y=4$ intersects the boundary of $\lambda$ at the point $(4,0)$.
\qedhere
\end{example}

To define $\Phi$, we need to first recursively define parameters $t_i^k$
for $1\leq i\leq k$. Our initial condition is that $t_1^1= d(1)$.  
Assume $t^{k-1}_1,\ldots, t_{k-1}^{k-1}$ has been previously determined.  
Let
\begin{equation}
\label{eqn:deltaDef}
\delta_i^{k}=D(\lambda^{(k)}, {\mathbf d}(k)-(t^{k-1}_{w^{(k)}(1)}+\ldots + t^{k-1}_{w^{(k)}(i)})) \text{\ \ for $i=1,\ldots, k-1$.}
\end{equation}

Suppose 
\begin{equation}
\label{eqn:ab-def}
\delta_i^k=a_i^k\times b_i^k \text{ rectangle.}
\end{equation}
    Let 
 \begin{equation}
\label{eqn:t1value}
t_{w^{(k)}(i)}^k= {\mathbf d}(k)-b_i^k-(t_{w^{(k)}(1)}^k+\ldots + t_{w^{(k)}(i-1)}^k) \text{\ for $i=1,\ldots, k-1$.}
\end{equation}
Finally, let
\begin{equation}
\label{eqn:tgenvalue} 
t_{w^{(k)}(k)}^k=t_k^k= {\mathbf d}(k)-(t_{w^{(k)}(1)}^k + \ldots t_{w^{(k)}(k-1)}^k).
\end{equation} 
Continue this procedure until $k=n$. 

Notice that by construction,
we have:
\begin{claim}
\label{claim:vhvg}
For $2\leq k\leq n$,
\[a_1^k\leq a_2^k\leq\cdots\leq a_{k-1}^k\]
and
\[b_1^k\geq b_2^k\geq\cdots\geq b_{k-1}^k.\]
\end{claim}

   We now
also fix parameters $s_i^k$ for $1\leq i\leq k-1$. Here we set
\begin{equation}
\label{eqn:s1value}
s_{w^{(k)}(1)}^k=a_1^k
\end{equation}
and  
\begin{equation}
\label{eqn:sgenvalue}
s_{w^{(k)}(i)}^k=a_i^k-a_{i-1}^k \text{ \ for \ } i=2,\ldots, k-1.
\end{equation}
These parameters are nonnegative integers, by Claim~\ref{claim:vhvg}.

\begin{claim}
\label{claim:recursion}
$t_{w^{(k)}(i)}^k+s_{w^{(k)}(i)}^k=t_{w^{(k)}(i)}^{k-1}$\ for $1\leq i<k\leq n$. 
\end{claim}
\begin{proof}
Fix $k$.  Our proof is by induction on
$i$. 

In the base case $i=1$, we have 
\begin{align*}
t_{w^{(k)}(1)}^k+s_{w^{(k)}(1)}^k&={\mathbf d}(k)-b_1^k+a_1^k & \text{(by
(\ref{eqn:t1value}) and (\ref{eqn:s1value}))}\\
&={\mathbf d}(k)-({\mathbf d}(k)-t_{w^{(k)}(1)}^{k-1}+a_1^k)+a_1^k & \text{(by (\ref{eqn:ab-def}))}\\
&=t_{w^{(k)}(1)}^{k-1}.
\end{align*}  Now assume 
\[t_{w^{(k)}(j)}^k+s_{w^{(k)}(j)}^k=t_{w^{(k)}(j)}^{k-1}\] 
holds for all $j<i$.  Then
\begin{align*} 
t_{w^{(k)}(i)}^k+s_{w^{(k)}(i)}^k&= {\mathbf d}(k)-b_i-(t_{w^{(k)}(1)}^k+\ldots + t_{w^{(k)}(i-1)}^k)+a_i^k-a_{i-1}^k \\
&= {\mathbf d}(k)-( {\mathbf d}(k)-(t^{k-1}_{w^{(k)}(1)}+\ldots + t^{k-1}_{w^{(k)}(i)})+a_i^k) \\
& \ \ \ \ -(t_{w^{(k)}(1)}^k+\ldots 
+ t_{w^{(k)}(i-1)}^k)+a_i^k-a_{i-1}^k & \text{(by (\ref{eqn:tgenvalue}) and (\ref{eqn:sgenvalue}))}\\
&=t_{w^{(k)}(i)}^{k-1}+(t_{w^{(k)}(1)}^{k-1}-t_{w^{(k)}(1)}^k)+\ldots\\
&\ \ \ \  +(t_{w^{(k)}(i-1)}^{k-1}-t_{w^{(k)}(i-1)}^k)-a_{i-1}^k\\
&=t_{w^{(k)}(i)}^{k-1}+s_{w^{(k)}(1)}^{k}+\ldots+s_{w^{(k)}(i-1)}^{k}-a_{i-1}^k & \text{(induction)}\\
&=t_{w^{(k)}(i)}^{k-1}. &\qedhere
\end{align*}
\end{proof}

\begin{claim}
\label{claim:laceParameters}
Let $\eta(\boldsymbol\lambda)$ be the equivalence class of a lacing diagram uniquely defined
by requiring that the number of strands:
\begin{itemize}
\item from $i$ to $j$ is 
$s^{j+1}_i$ for $1\leq i\leq j\leq n-1$;
\item 
from $i$ to $n$ is $t_i^n$ for $i=1\ldots n$.
\end{itemize}
Then:
\begin{enumerate}
\item $s_i^{k}(\eta(\boldsymbol\lambda))=s_i^{k}$
\item $t_j^{k}(\eta(\boldsymbol\lambda))=t_j^k$
\item ${\rm dim}(\eta(\boldsymbol\lambda))=\mathbf d$.
\end{enumerate} 
\end{claim}

\begin{proof}
(1) By hypothesis.

(2) 
By Claim~\ref{claim:recursion}, $t_i^k=t_i^{k+1}+s_i^{k+1}$.  Iterating, we obtain 
\begin{align*}
t_i^k&=t_i^{k+2}+s_i^{k+2}+s_i^{k+1}\\
&=\ldots\\
& =t_i^n+\sum_{\ell=k+1}^n s_i^\ell\\
&=t_i^n(\eta(\boldsymbol \lambda))+\sum_{\ell=k+1}^n s_i^\ell(\eta(\boldsymbol \lambda))&\text{(by hypothesis)}\\
&=t_i^k(\eta(\boldsymbol \lambda)).
\end{align*}

(3)  Let $\widetilde{\mathbf d}={\bf dim}(\eta(\boldsymbol\lambda))$.  By (2), we have \[ {\mathbf d}(k)=t_1^k+\ldots + t_k^k=t_1^k(\eta(\boldsymbol\lambda))+\ldots+t_k^k(\eta(\boldsymbol\lambda))=\widetilde{\mathbf d}(k). \qedhere\] 
\end{proof}

In view of Claim~\ref{claim:laceParameters},
we may disassemble each $\lambda^{(k)}$ as in Figure~\ref{figure:firstglue} to obtain rectangles of size 
\[s^k_{w^{(k)}(i)}(\eta(\boldsymbol\lambda))\times t^k_{w^{(k)}(j)}(\eta(\boldsymbol\lambda))
\text{\ (where $1\leq i<j\leq k$)}\] 
and partitions 
\[\nu_i^k\in P_{w^{(k)}(i)}^k(\eta(\boldsymbol\lambda)) \text{\ (where $1\leq i\leq k$).}\]  
That is, we have associated to $\boldsymbol\lambda$ a pair $(\boldsymbol\mu,\boldsymbol\nu)\in T(\eta(\boldsymbol\lambda))\subseteq T$.  This shows 
 $\Phi:S\rightarrow T$, as desired. 
 
 \noindent
 {\sf $\Phi$ is weight-preserving:} This is clear.

\begin{example}
\label{example:inverseMap}
Let $Q$ be an {\bf equioriented} quiver on $3$ vertices, i.e. all arrows point in the same direction.  
\begin{center}
\begin{tikzpicture}[x=.75cm,y=.75cm]
\filldraw [color=black,fill=black,thick](0,0)circle(.1);
\filldraw [color=black,fill=black,thick](1,0)circle(.1);
\filldraw [color=black,fill=black,thick](2,0)circle(.1);

\draw[ decoration={markings, mark=at position 0.6 with {\arrow[scale=2]{>}}}, postaction={decorate}](1,0) -- (2,0);
\draw[ decoration={markings, mark=at position 0.6 with {\arrow[scale=2]{>}}}, postaction={decorate}](0,0) -- (1,0);

\end{tikzpicture}
\end{center}
Then $\mathbf w_Q=(1,12,123)$.  Fix a dimension vector $ \mathbf d=(3,6,5)$ and partitions \[\lambda^{(1)}=(2,1), \lambda^{(2)}=(5,1), \text{\ and $\lambda^{(3)}=(3,3,2,1,1)$.}\]

\begin{center}
\ytableausetup{boxsize=1em}
\begin{tikzpicture}[x=1em,y=1em]
\node[ below right] at (-1.35,7.35){
\begin{ytableau}
\ &     \\
\\
\end{ytableau}};
\draw[very thick](-1,7)--(2,7);
\draw[very thick](-1,7)--(-1,0);

\node[ below right] at (3.65,7.35){
\begin{ytableau}
*(gray)&*(gray) &*(gray) &*(gray) &    \\
 \\
\end{ytableau}};
\draw[very thick](4,7)--(10.2,7);
\draw[very thick](4,7)--(4,0);
\draw[line width=.2em](4,7)rectangle (8.2,6);
\draw[very thick](7.1,7)--(8.1,6);
\filldraw [color=black,fill=black,thick](8.2,6)circle(.2);
\filldraw [color=black,fill=black,thick](7.15,7)circle(.2);

\node[ below right] at (12.65,7.35){
\begin{ytableau}
*(gray) &*(gray) &  \\
*(gray) &*(gray) &   \\
 *(gray)  &*(gray)   \\
 \\
 \\
 \end{ytableau}};
\draw[very thick](12,7)--(18,7);
\draw[very thick](13,7)--(13,0);

\draw[line width=.2em](13,7)--(16,7);
\filldraw [color=black,fill=black,thick](16.15,7)circle(.2);

\draw[line width=.2em](13,7)rectangle (15.1,3.9);
\draw[very thick](12,7)--(15,4);
\filldraw [color=black,fill=black,thick](12.05,7)circle(.2);
\filldraw [color=black,fill=black,thick](15.15,3.9)circle(.2);

\node[] at (.5,8){$t_1^1$};

\node[] at (9,8){$t_1^2$};
\node[] at (6,8){$t_2^2$};
\node[] at (3,6.5){$s_1^2$};

\node[] at (14,8){$t_3^3$};
\node[] at (15.6,8){$t_2^3$};
\node[] at (17.2,8){$t_1^3$};
\node[] at (12,5.4){$s_2^3$};
\end{tikzpicture}
\end{center}
Then 
\[\delta_1^2=D(\lambda^{(2)},6-3)=1\times 4
\text{ rectangle, $t_1^2=2$, and $t_2^2=4$.}\]  From this, we have \[\delta_1^3=D(\lambda^{(3)},5-2)=0\times 3
\text{\ and $\delta_2^3=D(\lambda^{(3)},5-2-4)=3\times 2$ rectangles.}\]  
So $t_1^3=2$, $t_2^3=1$, and $t_3^3=2$.  This corresponds to $\eta(\boldsymbol \lambda)=[\mathcal L]$ where

\begin{center}
\begin{tikzpicture}[x=1em,y=1em]
\node[] at (-1.5,1.5){$\mathcal L=$};

\filldraw [color=black,fill=black,thick](0,2)circle(.1);
\filldraw [color=black,fill=black,thick](0,1)circle(.1);
\filldraw [color=black,fill=black,thick](0,0)circle(.1);

\filldraw [color=black,fill=black,thick](1,5)circle(.1);
\filldraw [color=black,fill=black,thick](1,4)circle(.1);
\filldraw [color=black,fill=black,thick](1,3)circle(.1);
\filldraw [color=black,fill=black,thick](1,2)circle(.1);
\filldraw [color=black,fill=black,thick](1,1)circle(.1);
\filldraw [color=black,fill=black,thick](1,0)circle(.1);

\filldraw [color=black,fill=black,thick](2,4)circle(.1);
\filldraw [color=black,fill=black,thick](2,3)circle(.1);
\filldraw [color=black,fill=black,thick](2,2)circle(.1);
\filldraw [color=black,fill=black,thick](2,1)circle(.1);
\filldraw [color=black,fill=black,thick](2,0)circle(.1);

\draw[color=black,fill=black, thick](0,0)--(2,0);
\draw[color=black,fill=black, thick](0,1)--(2,1);
\draw[color=black,fill=black, thick](1,2)--(2,2);
\end{tikzpicture}
\end{center}
Alternatively, if $Q$ is {\bf bipartite}, that is adjacent arrows point in opposite directions, then $\mathbf w_Q=(1,12,213)$. 
\begin{center}
\begin{tikzpicture}[x=.75cm,y=.75cm]
\filldraw [color=black,fill=black,thick](0,0)circle(.1);
\filldraw [color=black,fill=black,thick](1,0)circle(.1);
\filldraw [color=black,fill=black,thick](2,0)circle(.1);

\draw[ decoration={markings, mark=at position 0.6 with {\arrow[scale=2]{<}}}, postaction={decorate}](1,0) -- (2,0);
\draw[ decoration={markings, mark=at position 0.6 with {\arrow[scale=2]{>}}}, postaction={decorate}](0,0) -- (1,0);

\end{tikzpicture}
\end{center}
 Keeping the same dimension vector and partitions $\lambda^{(k)}$  gives the following.
\begin{center}
\ytableausetup{boxsize=1em}
\begin{tikzpicture}[x=1em,y=1em]
\node[ below right] at (-1.35,7.35){
\begin{ytableau}
\ &     \\
\\
\end{ytableau}};
\draw[very thick](-1,7)--(2,7);
\draw[very thick](-1,7)--(-1,0);

\node[ below right] at (3.65,7.35){
\begin{ytableau}
*(gray)&*(gray) &*(gray) &*(gray) &    \\
 \\
\end{ytableau}};
\draw[very thick](4,7)--(10.2,7);
\draw[very thick](4,7)--(4,0);
\draw[line width=.2em](4,7)rectangle (8.2,6);
\draw[very thick](7.1,7)--(8.1,6);
\filldraw [color=black,fill=black,thick](8.2,6)circle(.2);
\filldraw [color=black,fill=black,thick](7.15,7)circle(.2);

\node[ below right] at (12.65,7.35){
\begin{ytableau}
*(gray) &*(gray) &*(gray)  \\
*(gray) &*(gray) & *(gray)  \\
 *(gray)  &*(gray)   \\
 \\
 \\
 \end{ytableau}};
\draw[very thick](12,7)--(18,7);
\draw[very thick](13,7)--(13,0);

\draw[line width=.2em](13,7)rectangle(16.2,5);
\filldraw [color=black,fill=black,thick](14.1,7)circle(.2);
\filldraw [color=black,fill=black,thick](16.2,4.95)circle(.2);
\draw[very thick](14,7)--(16.1,5);

\draw[line width=.2em](13,7)rectangle (15.1,3.9);
\draw[very thick](12,7)--(15,4);
\filldraw [color=black,fill=black,thick](12.05,7)circle(.2);
\filldraw [color=black,fill=black,thick](15.15,3.9)circle(.2);

\node[] at (.5,8){$t_1^1$};

\node[] at (9,8){$t_1^2$};
\node[] at (6,8){$t_2^2$};
\node[] at (3,6.5){$s_1^2$};

\node[] at (14,8){$t_3^3$};
\node[] at (15.6,8){$t_1^3$};
\node[] at (17.2,8){$t_2^3$};
\node[] at (12,4.4){$s_1^3$};
\node[] at (12,6){$s_2^3$};
\end{tikzpicture}
\end{center}
As before, 
\[\delta_1^2=D(\lambda^{(2)},6-3)=1\times 4
\text{\ rectangle.}\]  
Consequently, 
\[\delta_1^3=D(\lambda^{(3)},5-4)=2\times 3
\text{\ and $\delta_2^3=D(\lambda^{(3)},5-4-2)=3\times 2$ rectangles.}\]  
This yields  $\eta(\boldsymbol\lambda)=[\mathcal L']$, where

\begin{center}
\begin{tikzpicture}[x=1em,y=1em]
\node[] at (-1.5,1.5){$\mathcal L'=$};

\filldraw [color=black,fill=black,thick](0,2)circle(.1);
\filldraw [color=black,fill=black,thick](0,1)circle(.1);
\filldraw [color=black,fill=black,thick](0,0)circle(.1);

\filldraw [color=black,fill=black,thick](1,5)circle(.1);
\filldraw [color=black,fill=black,thick](1,4)circle(.1);
\filldraw [color=black,fill=black,thick](1,3)circle(.1);
\filldraw [color=black,fill=black,thick](1,2)circle(.1);
\filldraw [color=black,fill=black,thick](1,1)circle(.1);
\filldraw [color=black,fill=black,thick](1,0)circle(.1);

\filldraw [color=black,fill=black,thick](2,4)circle(.1);
\filldraw [color=black,fill=black,thick](2,3)circle(.1);
\filldraw [color=black,fill=black,thick](2,2)circle(.1);
\filldraw [color=black,fill=black,thick](2,1)circle(.1);
\filldraw [color=black,fill=black,thick](2,0)circle(.1);

\draw[color=black,fill=black, thick](0,0)--(2,0);
\draw[color=black,fill=black, thick](1,1)--(2,1);
\draw[color=black,fill=black, thick](1,2)--(2,2);
\draw[color=black,fill=black, thick](0,1)--(1,3);
\end{tikzpicture}
\end{center}
\qed
\end{example}

It remains to establish:

\begin{claim} $\Phi$ and $\Psi$ are mutual inverses.
\label{claim:mutualInverses}
\end{claim}
\begin{proof}
Taking $\boldsymbol\lambda\in S$, we have $\Psi(\Phi(\boldsymbol\lambda))=\boldsymbol\lambda$, since $\Phi$ acts by cutting the $\lambda^{(k)}$'s into various pieces and $\Psi$ glues these shapes together into their original configurations.  Now given $(\boldsymbol\mu,\boldsymbol\nu)\in T(\eta)$, let $\boldsymbol \lambda:=\Psi(\boldsymbol\mu,\boldsymbol\nu)$.  We must argue $\eta=\eta(\boldsymbol \lambda)$.  If so, $\Phi(\Psi(\boldsymbol\mu,\boldsymbol\nu))=(\boldsymbol\mu,\boldsymbol\nu).$

Since $\boldsymbol\lambda=\Psi(\boldsymbol\mu,\boldsymbol\nu)$ and $(\boldsymbol\mu,\boldsymbol\nu)\in T(\eta)$, each $\lambda^{(k)}$ contains a rectangle  
\begin{equation}
\label{eqn:epsilonRectangle}
\epsilon_j^k=\left(\sum_{i=1}^j s_{w^{(k)}(i)}^{k}(\eta)\right)\times \left(\sum_{i=j+1}^{k}t_{w^{(k)}(i)}^k(\eta)\right)
\end{equation}
 for all $1\leq j<k$ as in Figure~\ref{figure:firstglue}.  

By definition, ${\bf dim}(\eta)= \mathbf d$.  Then it follows
\[\sum_{i=j+1}^{k}t_{w^{(k)}(i)}^k(\eta)= {\mathbf d}(k)-\left(\sum_{i=1}^{j}t_{w^{(k)}(i)}^k(\eta)\right).
\]  From the definitions, $t_i^k(\eta)+s_i^k(\eta)=t_i^{k-1}(\eta)$.  So substituting 
 we have   
\begin{equation} 
\label{eqn:subs}
\sum_{i=j+1}^{k}t_{w^{(k)}(i)}^k(\eta)= {\mathbf d}(k)-\sum_{i=1}^{j}t_{w^{k}(i)}^{k-1}(\eta)+\sum_{i=1}^j s_{w^{(k)}(i)}^{k}(\eta).
\end{equation}

Substitution of (\ref{eqn:subs}) into (\ref{eqn:epsilonRectangle}) yields
\[\epsilon_j^k=s\times (s+ {\mathbf d}(k)-\sum_{i=1}^{j}t_{w^{k}(i)}^{k-1}(\eta)) \] contained in $\lambda^{(k)}$ (where $s=\sum_{i=1}^js_i^k(\eta)$). In particular, by construction, the bottom right corner of $\epsilon_j^k$ intersects the boundary of $\lambda^{(k)}$ (see Figure~\ref{figure:firstglue}), i.e. $s$ is the maximum value for which $\epsilon_j^k\subseteq \lambda^{(k)}$.  So by the definition of a Durfee rectangle,
\[\epsilon_j^k=D(\lambda^{(k)},{\mathbf d}(k)-\sum_{i=1}^{j}t_{w^{k}(i)}^{k-1}(\eta)).\]

By (\ref{eqn:deltaDef}) and Claim~\ref{claim:laceParameters} part (2), \[\delta_j^k=D(\lambda^{(k)}, {\mathbf d}(k)-\sum_{i=1}^{j}t_{w^{k}(i)}^{k-1}(\eta(\boldsymbol \lambda))).\] 
   Then if 
   \begin{equation}
   \label{eqn:equalt}
   \sum_{i=1}^{j}t_{w^{k}(i)}^{k-1}(\eta)=\sum_{i=1}^{j}t_{w^{k}(i)}^{k-1}(\eta(\boldsymbol \lambda))),
   \end{equation} it follows that $\delta_j^k=\epsilon_j^k$ since both are Durfee rectangles with the \emph{same} parameter, and are maximal among such rectangles.
   
   For $k=2$, since $t_1^1(\eta)=\mathbf d(1)=t_1^1(\eta(\boldsymbol \lambda))$, then 
   \begin{align*}\delta_1^2&=D(\lambda^{(2)},\mathbf d(2)-t_1^1(\eta))\\
   &=D(\lambda^{(2)}-\mathbf d(2)-t_1^1(\eta(\boldsymbol \lambda)))\\
   &=\epsilon_1^2,
   \end{align*}
   so the Durfee rectangles agree.
   Assume $\delta_j^{k-1}=\epsilon_j^{k-1}$ for all $1\leq j<k-1$.  Then in particular, $t_i^{k-1}(\eta)=t_i^{k-1}(\eta(\boldsymbol \lambda))$ for all $1\leq i\leq k-1$.  So by (\ref{eqn:equalt}), $\delta_j^k=\epsilon_j^k$.

Therefore, $s_i^k(\eta)=s_i^k(\eta(\boldsymbol \lambda))$ for all $1\leq i<k\leq n$ and $t_i^k(\eta)=t_i^k(\eta(\boldsymbol\lambda))$ for $1\leq i\leq k\leq n$. Hence $\eta=\eta(\boldsymbol\lambda)$.
\end{proof}
\end{proof}

Actually, the proof of Theorem~\ref{thm:bij} implies
an enriched form of Theorem~\ref{thm:main}. 

Let 
\[(z;q)_k=(1-qz)(1-q^2z)\cdots (1-q^kz).\]
Also, for a lace equivalence class $\eta$, let
${\tt leftstrands}_{\eta}(j)$ be the number of strands that
terminate at column $j$ in some (equivalently any)
lace diagram ${\mathcal L}\in \eta$. That is,
\begin{equation}
\label{eqn:leftstrands}
{\tt leftstrands}_{\eta}(j)=\sum_{i=1}^{j}s_i^{j+1}(\eta).
\end{equation}

\begin{corollary}[of Theorem~\ref{thm:bij}]
\label{cor:enriched}
\begin{equation}
\label{eqn:enriched}
\prod_{k=1}^{n} \frac{1}{(z;q)_{ {\mathbf d}(k)}}=\sum_{\eta}
q^{r_{\mathbf w}(\eta)} \prod_{k=1}^n 
z^{{\tt leftstrands}_{\eta}(k-1)}
\frac{1}{(z;q)_{t_k^k(\eta)}}
\prod_{i=1}^{k-1}{t_i^k(\eta)+s_i^k(\eta)\brack s_i^k(\eta)}_q.
\end{equation}
\end{corollary}
\begin{proof}
The lefthand side of (\ref{eqn:enriched}) is the generating series for $S$ with respect to the weight
that uses $q$ to mark the number of boxes and
$z$ to mark length of the partitions involved.
Now, suppose $\lambda^{(k)}$ is a partition of
${\boldsymbol \lambda}\in S$ of length $\ell$.
Under the indicated decomposition of Figure~\ref{figure:firstglue},
\[\ell=\ell(\nu_k^k)+\sum_{i=1}^{k-1}s_{w^{(k)}(i)}^k=\ell(\nu_k^k)+{\tt leftstrands}_{\eta(k-1)},\]
where the second equality holds by
(\ref{eqn:leftstrands}) and reordering terms. 
Here $\ell(\nu_k^k)$ is the length of $\nu_k^k$.
The corollary follows immediately from this and
Theorem~\ref{thm:bij} combined.
\end{proof}

Theorem~\ref{thm:main} is therefore the $z=1$
case of Corollary~\ref{cor:enriched}.
By analysis as in Example~\ref{exa:classical}, we obtain, in a special case this Durfee square identity:
\[\frac{1}{(z;q)_k}=\sum_{j=0}^{\infty}
z^j q^{j^2}{k\brack j}_q\frac{1}{(z;q)_j}.\]
In addition, following the argument of the Introduction,
from Corollary~\ref{cor:enriched} one can thereby
deduce an enriched form of M.~Reineke's identity.

\section{Proof of Theorem~\ref{prop:codim}}
\label{section:Reineke}

First we recall some more background on quiver representations.  Given ${\sf V}$ and ${\sf W}$ an {\bf extension} of $\sf V$ by $\sf W$ is a short exact sequence of morphisms
\[0\rightarrow \sf W\rightarrow \sf E\rightarrow \sf V\rightarrow 0.\]
Two extensions are {\bf equivalent} if the following diagram commutes:

\begin{center}
\includegraphics{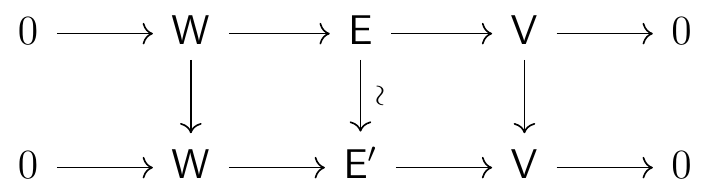}
\end{center}
Write ${\rm Ext}^1(\sf V,\sf W)$ for the space of extensions of $\sf V$ by $\sf W$ up to equivalence.

Each quiver has an associated {\bf Euler form} \[\chi_Q:\mathbb N^{Q_0}\times \mathbb N^{Q_0}\rightarrow \mathbb Z,\] 
defined by 
\begin{equation}
\label{eqn:eulerform}
\chi_Q({\mathbf d_1},{\mathbf d_2})=\sum_{x\in Q_0} {\mathbf d_1}(x){\mathbf d_2}(x)-\sum_{a\in Q_1} {\mathbf d_1}(t(a)){\mathbf d_2}(h(a)).
\end{equation}    Given representations $\sf V$ and $\sf W$ of $Q$, use the abbreviation:
\[\chi_Q({\sf V},{\sf W}):=\chi_Q({\bf dim}{\sf V},{\bf dim}{\sf W}).\]  
The Euler form relates morphisms and extensions as follows:
 \begin{equation}
\label{eqn:eulerhom}
\chi_Q(\sf V,\sf W)={\rm dim}{\rm Hom}(\sf V,\sf W)-{\rm dim}{\rm Ext}^1(\sf V,\sf W),
\end{equation}  
(see \cite[Corollary 1.4.3]{brion2008representations}).

Below, we let $a_x$ to refer to the arrow of the quiver whose left vertex is $x$.
Consider pairs of intervals $(I,J)$ of the following three types:

\begin{enumerate}
\item[(I)] $I=[w,x-1]$ and $J=[x,z]$ with $w<x\leq z$

\begin{tikzpicture}[x=.75cm,y=.75cm]
\filldraw [color=black,fill=black,thick](-1,1)circle(.1);
\filldraw [color=black,fill=black,thick](2,1)circle(.1);
\filldraw [color=black,fill=black,thick](3,0)circle(.1);
\filldraw [color=black,fill=black,thick](6,0)circle(.1);
\draw[dashed] (-1,1)--(2,1);
\draw[dashed] (3,0)--(6,0);
\node[] at (3,-.5) {$x$};
\node[] at (6,-.5) {$z$};
\node[] at (-1,1.5) {$w$};
\node[] at (2,1.5){$x-1$};
\end{tikzpicture}
\item[(II)] $I=[w,y]$ and $J=[x,z]$ with $w<x\leq y<z$ and the arrows $a_{x-1}$ and $a_y$ point in the same direction, e.g.,

\begin{tikzpicture}[x=.75cm,y=.75cm]
\filldraw [color=black,fill=black,thick](-1,1)circle(.1);
\filldraw [color=black,fill=black,thick](3,1)circle(.1);
\filldraw [color=black,fill=black,thick](2,0)circle(.1);
\filldraw [color=black,fill=black,thick](6,0)circle(.1);
\draw[dashed] (-1,1)--(3,1);
\draw[dashed] (2,0)--(6,0);
\node[] at (2,-.5) {$x$};
\node[] at (6,-.5) {$z$};
\node[] at (-1,1.5) {$w$};
\node[] at (3,1.5){$y$};
\filldraw [color=black,fill=black,thick](2,1)circle(.1);
\filldraw [color=black,fill=black,thick](1,1)circle(.1);
\filldraw [color=black,fill=black,thick](3,0)circle(.1);
\filldraw [color=black,fill=black,thick](4,0)circle(.1);
\draw[ decoration={markings, mark=at position 0.6 with {\arrow[scale=2]{>}}}, postaction={decorate}](1,1) -- (2,1);
\draw[ decoration={markings, mark=at position 0.6 with {\arrow[scale=2]{>}}}, postaction={decorate}](3,0) -- (4,0);
\end{tikzpicture}

\item[(III)] $I=[x,y]$ and $J=[w,z]$ with $w<x\leq y<z$ and the arrows $a_{x-1}$ and $a_{y}$ point in different directions, e.g.,

\begin{tikzpicture}[x=.75cm,y=.75cm]
\filldraw [color=black,fill=black,thick](2,1)circle(.1);
\filldraw [color=black,fill=black,thick](3,1)circle(.1);
\filldraw [color=black,fill=black,thick](-1,0)circle(.1);
\filldraw [color=black,fill=black,thick](2,0)circle(.1);
\filldraw [color=black,fill=black,thick](4,0)circle(.1);
\filldraw [color=black,fill=black,thick](3,0)circle(.1);
\filldraw [color=black,fill=black,thick](1,0)circle(.1);
\filldraw [color=black,fill=black,thick](6,0)circle(.1);
\draw[dashed] (-1,0)--(6,0);
\draw[dashed] (2,1)--(3,1);
\node[] at (-1,-.5) {$w$};
\node[] at (6,-.5) {$z$};
\node[] at (2,1.5) {$x$};
\node[] at (3,1.5){$y$};
\draw[ decoration={markings, mark=at position 0.6 with {\arrow[scale=2]{>}}}, postaction={decorate}](1,0) -- (2,0);
\draw[ decoration={markings, mark=at position 0.6 with {\arrow[scale=2]{<}}}, postaction={decorate}](3,0) -- (4,0);
\end{tikzpicture}
\end{enumerate}

Let \[{\tt ConditionStrands}=\{(I,J):  \text{ $(I,J)$ satisfies (I), (II), or (III)}\}.\] 
 We also let \[{\tt StrandPairs}=\{(I,J)=([x_1,x_2],[y_1,y_2]:x_2\leq y_2)\}.\]  
(From the definitions (I)-(III), it follows that ${\tt ConditionStrands}\subset {\tt StrandPairs}$.)

\begin{claim}
\label{claim:interval}
Fix intervals $I$ and $J$.  If $[x,y]\subseteq I,J$ then 
\begin{equation}
\label{eqn:eulerrestrict}
\sum_{i=x}^y{\mathbf d_I}(i){\mathbf d_J}(i)-\sum_{i=x}^{y-1}{\mathbf d_I}(t(a_i)){\mathbf d_J}(h(a_i))=1
\end{equation}
\end{claim}
\begin{proof}
Since $[x,y]\subseteq I,J$, ${\mathbf d_I}(i)={\mathbf d_J}(i)=1$ for all $i\in[x,y]$.  Therefore, 
\begin{equation}
\label{eqn:vert}
\sum_{i=x}^y{\mathbf d_I}(i){\mathbf d_J}(i)=y-x+1.
\end{equation}

Regardless of the orientation of $a_i$, if $i\in[x,y-1]$ then $t(a_i),h(a_i)\in [x,y]$.  Because $[x,y]\subseteq I,J$, we have ${\mathbf d_I}(t(a_i))={\mathbf d_J}(h(a_i))=1$.  So 
\begin{equation}
\label{eqn:arrow}
\sum_{i=x}^{y-1}{\mathbf d_I}(t(a_i)){\mathbf d_J}(h(a_i))=(y-1)-x+1.
\end{equation}
Subtracting (\ref{eqn:arrow}) from (\ref{eqn:vert}) gives (\ref{eqn:eulerrestrict}).
\end{proof}

\begin{claim}
\label{claim:eulernegative} Let $(I,J)\in {\tt StrandPairs}$.  Then 
\[(I,J)\in {\tt ConditionStrands} \iff \chi_Q({\sf V}_I,{\sf V}_J)<0 \text{\ or $\chi_Q({\sf V}_J,{\sf V}_I)<0$.}\] 
Moreover,
\[(I,J)\in {\tt ConditionStrands} \Rightarrow \chi_Q(V_I,V_J)=-1  \text{\ or $\chi(V_J,V_I)=-1$.}\]
\end{claim}
\begin{proof}
Throughout, given an interval $I$, write $\mathbf d_I$ for the dimension vector of ${\sf V}_I$.  Applying (\ref{eqn:eulerform}), the definition of the Euler form, \[\chi_Q({\sf V}_I,{\sf V}_J)=\chi_Q({\mathbf d_I},{\mathbf d_J})=\sum_{i=1}^n{\mathbf d_I}(i){\mathbf d_J}(i)-\sum_{i=1}^{n-1}{\mathbf d_I}(t(a_i)){\mathbf d_J}(h(a_i)).\]  
We analyze this expression repeatedly throughout our argument.

$(\Rightarrow)$ By direct computation, we will show if $(I,J)\in {\tt ConditionStrands}$ then 
\[\chi_Q({\sf V}_I,{\sf V}_J)=-1 \text{ or }\chi_Q({\sf V}_J,{\sf V}_I)=-1,\]
which is the last assertion of the claim.

\noindent Case 1: $(I,J)=([w,x-1],[x,z])$ is of type (I).

\noindent Subcase i: $a_{x-1}$ points to the right.
\begin{align*}
\chi_Q({\sf V}_I,{\sf V}_J)&=\sum_{i=1}^n{\mathbf d_I}(i){\mathbf d_J}(i)-\sum_{i=1}^{n-1}{\mathbf d_I}(t(a_i)){\mathbf d_J}(h(a_i))\\
&=-\sum_{i=1}^{n-1}{\mathbf d_I}(t(a_i)){\mathbf d_J}(h(a_i)) \text{\ \ \ \ (since $I\cap J=\emptyset$)}\\
&=-{\mathbf d_I}(t(a_{x-1})){\mathbf d_J}(h(a_{x-1}))\\
&=-{\mathbf d_I}(x-1){\mathbf d_J}(x)\\
&=-1
\end{align*}
\noindent Subcase ii: $a_{x-1}$ points to the left.

Let $Q^{\rm op}$ be the quiver obtained by reversing the direction of all arrows in $Q$.  Then $\chi_Q({\mathbf d_J},{\mathbf d_I})=\chi_{Q^{\rm op}}({\mathbf d_I},{\mathbf d_J})$. Therefore, 
\[
\chi_Q({\sf V}_J,{\sf V}_I)=\chi_Q({\mathbf d_J},{\mathbf d_I})=\chi_{Q}^{\rm op}({\mathbf d_I},{\mathbf d_J})=-1\]
by Subcase 1.i.

\noindent Case 2: $(I,J)=([w,y],[x,z])$ is of type (II).

\noindent Subcase i: $a_{x-1}$ and $a_{y}$ point to the right.
\begin{align*}
\chi_Q({\sf V}_I,{\sf V}_J)&=\sum_{i=x}^y{\mathbf d_I}(i){\mathbf d_J}(i)-\sum_{i=x-1}^{y}{\mathbf d_I}(t(a_i)){\mathbf d_J}(h(a_i))\\
&=\left(\sum_{i=x}^y{\mathbf d_I}(i){\mathbf d_J}(i)-\sum_{i=x}^{y-1}{\mathbf d_I}(t(a_i)){\mathbf d_J}(h(a_i))\right)-{\mathbf d_I}(t(a_{x-1})){\mathbf d_J}(h(a_{x-1}))\\
&\quad-{\mathbf d_I}(t(a_y)){\mathbf d_J}(h(a_y))\\
&=1-{\mathbf d_I}(t(a_{x-1})){\mathbf d_J}(h(a_{x-1}))-{\mathbf d_I}(t(a_y)){\mathbf d_J}(h(a_y)) \quad \text{(Claim~\ref{claim:interval})}\\
&=1-{\mathbf d_I}(x-1){\mathbf d_J}(x)-{\mathbf d_I}(y){\mathbf d_J}(y+1)\\
&=-1
\end{align*}

\noindent Subcase ii: $a_{x-1}$ and $a_{y}$ point to the left.

$\chi_Q({\sf V}_J,{\sf V}_I)=-1$ by the $Q^{\rm op}$ argument, as in Subcase 1.i.

\noindent Case 3: $(I,J)=([x,y],[y,z])$ is of type (III).

\noindent Subcase i: $a_{x-1}$ points right and $a_{y}$ points left.
\begin{align*}
\chi_Q({\sf V}_I,{\sf V}_J)&=\sum_{i=x}^y{\mathbf d_I}(i){\mathbf d_J}(i)-\sum_{i=x-1}^{y}{\mathbf d_I}(t(a_i)){\mathbf d_J}(h(a_i))\\
&=\left(\sum_{i=x}^y{\mathbf d_I}(i){\mathbf d_J}(i)-\sum_{i=x}^{y-1}{\mathbf d_I}(t(a_i)){\mathbf d_J}(h(a_i))\right)-{\mathbf d_I}(t(a_{x-1})){\mathbf d_J}(h(a_{x-1}))\\
&\quad-{\mathbf d_I}(t(a_y)){\mathbf d_J}(h(a_y))\\
&=1-{\mathbf d_I}(t(a_{x-1})){\mathbf d_J}(h(a_{x-1}))-{\mathbf d_I}(t(a_y)){\mathbf d_J}(h(a_y)) \quad \text{ (Claim~\ref{claim:interval})}\\
&=1-{\mathbf d_I}(x-1){\mathbf d_J}(x)-{\mathbf d_I}(y-1){\mathbf d_J}(y)\\
&=-1
\end{align*}
\noindent Subcase ii: $a_{x-1}$ points left and $a_{y}$ points right.

$\chi_Q({\sf V}_J,{\sf V}_I)=-1$ by the $Q^{\rm op}$ argument, as in Subcase 1.i.

$(\Leftarrow)$
Let $(I,J)=([x_1,x_2],[y_1,y_2])\in{\tt StrandPairs}$ and first assume $\chi_Q({\sf V}_I,{\sf V}_J)<0$.  

\noindent Case 1: $I\cap J=\emptyset$.
Then ${\mathbf d_I}(i)=0$ or ${\mathbf d_J}(i)=0$ for all $i\in[1,n]$ and so
\[\chi_Q({\mathbf d_I},{\mathbf d_J})=-\sum_{i=1}^{n-1}{\mathbf d_I}(t(a_i)){\mathbf d_J}(h(a_i)).\]
Since  $\chi_Q({\mathbf d_I},{\mathbf d_J})<0$ there must exist an arrow $a_i$ with $t(a_i)\in [x_1,x_2]$ and $h(a_i)\in [y_1,y_2]$.  Then $i=x_2$, $a_i$ points to the right, and $y_1=x_2+1$.  This implies $(I,J)$ is of type (I).  

\noindent Case 2: Assume $I\cap J\neq \emptyset$.  Since we assume $x_2\leq y_2$ \[I\cap J=[x_1,x_2]\cap[y_1,y_2]=[z,x_2]\] 
where $z\in\{x_1,y_1\}$.  Then
\begin{align*}
\chi_Q({\mathbf d_I},{\mathbf d_J})&=\sum_{i=1}^n {\mathbf d_I}(i){\mathbf d_J}(i)-\sum_{i=1}^{n-1}{\mathbf d_I}(t(a_i)){\mathbf d_J}(h(a_i))\\
&=\sum_{i=z}^{x_2}{\mathbf d_I}(i){\mathbf d_J}(i)-\sum_{i=z-1}^{x_2} {\mathbf d_I}(t(a_i)){\mathbf d_J}(h(a_i)) \text{ \ \ \ (Claim~\ref{claim:interval})}\\
&=1-{\mathbf d_I}(t(a_{z-1})){\mathbf d_J}(h(a_{z-1}))-{\mathbf d_I}(t(a_{x_2})){\mathbf d_J}(h(a_{x_2})).
\end{align*}
Since $\chi_Q({\mathbf d_I},{\mathbf d_J})<0$, we must have \[{\mathbf d_I}(t(a_{z-1}))={\mathbf d_J}(h(a_{z-1}))={\mathbf d_I}(t(a_{x_2}))={\mathbf d_J}(h(a_{x_2}))=1.\]  
Therefore, 
\begin{equation}
\label{eqn:tail}
t(a_{z-1}),t(a_{x_2})\in I=[x_1,x_2]
\end{equation} and 
\begin{equation}
\label{eqn:head}
h(a_{z-1}),h(a_{x_2})\in J=[y_1,y_2].
\end{equation}
If an arrow $a_i$ points to the right, then $h(a_i)=i+1$ and $t(a_i)=i$.  If $a_i$ points left, $h(a_i)=i$ and $t(a_i)=i+1$.  We proceed by analyzing the direction of $a_{x_2}$ and $a_{z-1}$  . 
First consider $a_{x_2}$.  If $a_{x_2}$ points left, then $t(a_{x_2})=x_2+1$ and so  $x_2+1\in[x_1,x_2]$, which is a contradiction.  Therefore, we may assume $a_{x_2}$ points right.  

Now consider the direction of $a_{z-1}$.  

If $a_{z-1}$ points to the right, then $t(a_{z-1})=z-1\in [x_1,x_2]$ by $(\ref{eqn:tail})$ and so $z>x_1$.  Since $z\in\{x_1,y_1\}$, we must have $z=y_1$.

\begin{center}
\begin{tikzpicture}[x=.75cm,y=.75cm]
\filldraw [color=black,fill=black,thick](-1,1)circle(.1);
\filldraw [color=black,fill=black,thick](3,1)circle(.1);
\filldraw [color=black,fill=black,thick](2,0)circle(.1);
\filldraw [color=black,fill=black,thick](6,0)circle(.1);
\draw[dashed] (-1,1)--(3,1);
\draw[dashed] (2,0)--(6,0);
\node[] at (2,-.5) {$z=y_1$};
\node[] at (6,-.5) {$y_2$};
\node[] at (-1,1.5) {$x_1$};
\node[] at (3,1.5){$x_2$};
\filldraw [color=black,fill=black,thick](2,1)circle(.1);
\filldraw [color=black,fill=black,thick](1,1)circle(.1);
\filldraw [color=black,fill=black,thick](3,0)circle(.1);
\filldraw [color=black,fill=black,thick](4,0)circle(.1);
\draw[ decoration={markings, mark=at position 0.6 with {\arrow[scale=2]{>}}}, postaction={decorate}](1,1) -- (2,1);
\draw[ decoration={markings, mark=at position 0.6 with {\arrow[scale=2]{>}}}, postaction={decorate}](3,0) -- (4,0);
\end{tikzpicture}
\end{center}
Therefore $(I,J)$ is of type (II).

 If $a_{z-1}$ points left, now we have by (\ref{eqn:head})  $h(a_{z-1})=z-1\in[y_1,y_2]$.   Therefore $z-1>y_1$ and so $z\neq y_1$ which implies $z=x_1$. Hence we have:
\begin{center}
\begin{tikzpicture}[x=.75cm,y=.75cm]
\filldraw [color=black,fill=black,thick](2,1)circle(.1);
\filldraw [color=black,fill=black,thick](3,1)circle(.1);
\filldraw [color=black,fill=black,thick](-1,0)circle(.1);
\filldraw [color=black,fill=black,thick](2,0)circle(.1);
\filldraw [color=black,fill=black,thick](4,0)circle(.1);
\filldraw [color=black,fill=black,thick](3,0)circle(.1);
\filldraw [color=black,fill=black,thick](1,0)circle(.1);
\filldraw [color=black,fill=black,thick](6,0)circle(.1);
\draw[dashed] (-1,0)--(6,0);
\draw[dashed] (2,1)--(3,1);
\node[] at (-1,-.5) {$y_1$};
\node[] at (6,-.5) {$y_2$};
\node[] at (1.5,1.5) {$z=x_1$};
\node[] at (3,1.5){$x_2$};
\draw[ decoration={markings, mark=at position 0.6 with {\arrow[scale=2]{>}}}, postaction={decorate}](1,0) -- (2,0);
\draw[ decoration={markings, mark=at position 0.6 with {\arrow[scale=2]{<}}}, postaction={decorate}](3,0) -- (4,0);
\end{tikzpicture}
\end{center}
So $(I,J)$ is of type (III).

By near identical arguments, $\chi_Q({\mathbf d_J},{\mathbf d_I})$ is negative when 
\begin{enumerate}
\item $a_{z-1}$ and $a_{x_2}$ both point left, $z=y_1$, and $x_2<y_2$; i.e., $(I,J)$ is of type (II)
\item $a_{z-1}$ points right, $a_{x_2}$ points left, $z=x_1$ and $x_2<y_2$ so $(I,J)$ is of type (III).
\end{enumerate}

\end{proof}

\begin{proposition}
\label{prop:first1}
\[{\rm codim}_\mathbb C\eta=\sum_{(I,J)\in{\tt ConditionStrands}}m_Im_J\]
\end{proposition}

\begin{proof}

There exists a total order on $\Phi^+$
\begin{equation}
\label{eqn:totalordertalk}
{\rm Hom}({\sf V}_{I},{\sf V}_{J}) \text{ and } {\rm Ext}^1({\sf V}_{J},{\sf V}_{I})=0
\text{\ whenever $I<J$ and $I\neq J$,}
\end{equation}
(see \cite{reineke2001feigin}, Section 2).
Using this ordering and (\ref{eqn:eulerhom}), it follows that
\begin{equation}
\label{eqn:eulerorder}
\text{ if $I<J$, then $\chi_Q({\sf V}_{I},{\sf V}_{J})\leq 0$ and  $\chi_Q({\sf V}_{J},{\sf V}_{I})\geq 0$.}
\end{equation}

Voigt's Lemma  (see \cite[Lemma 2.3]{ringel1980rational}) asserts 
\[{\rm codim}_\mathbb C \eta={\rm dim}{\rm Ext}^1({\sf V}_\eta,{\sf V}_\eta).\]
  Furthermore, indecomposables for Dynkin quivers have no self extensions, that is \[{\rm Ext}^1({\sf V}_I,{\sf V}_I)=0 \text{ for all } I\in \Phi^+.\]  So writing \[{\sf V}_\eta\cong \bigoplus_{I\in \Phi^+} {\sf V}_{I}^{\oplus m_I}\] as a finite direct sum of indecomposables, we have

\[{\rm Ext}^1({\sf V}_\eta,{\sf V}_\eta)\cong \bigoplus_{I<J} {\rm Ext}^1({\sf V}_{I},{\sf V}_{J})^{\oplus m_I m_J}\]
and so
\[{\rm codim}_\mathbb C\eta=\sum_{I<J} m_Im_J {\rm dim} {\rm Ext}^1({\sf V}_{I},{\sf V}_{J}),\] (see \cite{rimanyi2013cohomological}).  
Combining (\ref{eqn:eulerhom}) and (\ref{eqn:totalordertalk}) gives 
\begin{align}
\label{eqn:korea123}
{\rm codim}_\mathbb C\eta&=-\sum_{I<J} m_Im_J \chi_Q({\sf V}_{I},{\sf V}_{J}).
\end{align}

We will now re-express (\ref{eqn:korea123}).   
Let \[S=\{(I,J):I<J \text{ and } \chi_Q({\sf V}_{I},{\sf V}_{J})< 0  \},\] 
\[S_1=\{(I,J)=([x_1,x_2],[y_1,y_2]): (I,J)\in S \text{ and } x_2\leq y_2\}, \text{ and }\] 
\[S_2=\{(I,J)=([x_1,x_2],[y_1,y_2]):(I,J)\in S  \text{ and } x_2> y_2\}.\] 
Trivially, $S=S_1	\sqcup S_2$.
Let \[\widetilde{S_2}=\{(J,I):(I,J)\in S_2\}.\]  
\begin{claim}
\label{claim:toronto456}
${\tt ConditionStrands}=S_1\sqcup \widetilde{S_2}.$  
\end{claim}
\begin{proof}
$S_1\cap \widetilde{S_2}=\emptyset$, since $(I,J)\in S_1$ implies $I<J$ and $(I,J)\in \widetilde{S_2}$ implies $I>J$.

$(\subseteq)$  If $(I,J)\in {\tt ConditionStrands}$, by Claim~\ref{claim:eulernegative},  $\chi_Q({\sf V}_{I},{\sf V}_{J})< 0$ or  $\chi_Q({\sf V}_{J},{\sf V}_{I})< 0$.  In the first case,
from the definition, $(I,J)\in S_1$.  In the second case, again by definition, $(J,I)\in S_2$, which implies $(I,J)\in \widetilde S_2$.

$(\supseteq)$
We have $S_1,\widetilde{S_2}\subseteq {\tt StrandPairs}$.
Thus by Claim~\ref{claim:eulernegative}, $S_1,\widetilde{S_2}\subseteq{\tt ConditionStrands}$.
\end{proof}
Continuing from (\ref{eqn:korea123}),
\begin{align*}
{\rm codim}_\mathbb C\eta&=-\sum_{(I,J)\in S}m_Im_J\chi_Q({\sf V}_{I},{\sf V}_{J})\\
&=-\sum_{(I,J)\in S_1}m_Im_J\chi_Q({\sf V}_{I},{\sf V}_{J})-\sum_{(I,J)\in S_2}m_Im_J\chi_Q({\sf V}_{I},{\sf V}_{J})\\
&=-\sum_{(I,J)\in S_1}m_Im_J\chi_Q({\sf V}_{I},{\sf V}_{J})-\sum_{(I,J)\in \widetilde{S_2}}m_Im_J\chi_Q({\sf V}_{J},{\sf V}_{I})\\
&=\sum_{(I,J)\in S_1}m_Im_J+\sum_{(I,J)\in \widetilde{S_2}}m_Im_J
\text{ \ \ (Claim~\ref{claim:eulernegative})}\\
&=\sum_{(I,J)\in{\tt ConditionStrands}}m_Im_J \text{ \ \ (Claim~\ref{claim:toronto456}),}
\end{align*}
as claimed.
\end{proof}

Let 
\begin{equation}
\label{eqn:boxDef}
{\tt BoxStrands}=\{([w^{(k)}(i),k-1],[w^{(k)}(j),\ell]):1\leq i<j\leq k\leq \ell \leq n)\}.  
\end{equation}
(By definition, if $(I,J)=([w^{(k)}(i),k-1],[w^{(k)}(j),\ell]\in {\tt BoxStrands}$ then $k-1\leq \ell$, and so $(I,J)\in {\tt StrandPairs}$.  Thus ${\tt BoxStrands}\subset {\tt StrandPairs}$.)

\begin{proposition}
\label{prop:second1}
\[r_{\mathbf w}(\eta)=\sum_{(I,J)\in {\tt BoxStrands}} m_{I}m_{J}.\]\end{proposition}
\begin{proof}
By definition (\ref{eqn:durfeeStatistic}), 
\[r_{\mathbf w}(\eta)=\sum_{k=2}^n\sum_{1\leq i<j\leq k} s_{w^{(k)}(i)}^k(\eta) t_{w^{(k)}(j)}^k(\eta).\]

By definition, $t_{w^{(k)}(j)}^k(\eta)$ counts the number of strands in $\eta$ starting at $w^{(k)}(j)$ and using a vertex in column $k$.  So 
\[t_{w^{(k)}(j)}^k(\eta)=\sum_{\ell=k}^n m_{[w^{(k)}(j),\ell]}.\]  
Also, 
\[s_{w^{(k)}(i)}^k(\eta)=m_{[w^{(k)}(i),k-1]}.\]  Making these substitutions,
\begin{align*}
r_{\mathbf w}(\eta)&=\sum_{k=2}^n\sum_{1\leq i<j\leq k} m_{[w^{(k)}(i),k-1]}\left(\sum_{\ell=k}^n m_{w^{(k)}(j),\ell}\right)\\
&=\sum_{1\leq i<j\leq k\leq \ell\leq n}  m_{[w^{(k)}(i),k-1]}m_{[w^{(k)}(j),\ell]}\\
&=\sum_{(I,J)\in{\tt BoxStrands}}m_Im_J. \qedhere
\end{align*}

\end{proof}

It remains to prove
\begin{lemma}
\label{lemma:bigfinale}
${\tt BoxStrands}={\tt ConditionStrands}$.
\end{lemma}

\begin{proof}

Let $(I,J)$ be as follows:
\begin{equation}
\label{eqn:inthisform}
(I,J):=([x,k-1],[y,\ell]), \text{\ with $x\neq y$, $k\leq \ell$.}
\end{equation}
\begin{claim}
All elements of {\tt BoxStrands} and {\tt ConditionStrands} may be written in the form (\ref{eqn:inthisform}).
\end{claim}
\begin{proof}
If 
\[([w^{(k)}(i),k-1],[w^{(k)}(j),\ell])\in {\tt Boxstrands},\] 
then 
\[w^{(k)}(i)\neq w^{(k)}(j) \text{\ and $k\leq \ell$.}\]  
Hence we are done here by setting $x=w^{(k)}(i)$ and $y=w^{(k)}(j)$.

 On the other hand, suppose 
 \[([x_1,x_2],[y_1,y_2])\in {\tt ConditionStrands}.\] By definition (I)-(III), $x_1\neq y_1$ and $x_2<y_2$.  So set $x=x_1$, $y=y_1$, $k=x_2+1$ and $\ell=y_2$.
\end{proof}

\begin{claim} 
\label{claim:trivIntersect}
Let $(I,J)$ be as in (\ref{eqn:inthisform}) and suppose $I\cap J=\emptyset$.  Then $(I,J)\in {\tt BoxStrands}$ if and only if $(I,J)\in {\tt ConditionStrands}$.
\end{claim}
\begin{proof}
If $(I,J)\in{\tt ConditionStrands}$, then by the disjointness hypothesis it must be of type (I), i.e. of the form \[([x,k-1],[k,\ell]).\]
Now, since $x\leq k-1$ and as $w^{(k)}\in \mathfrak S_k$ and $w^{(k)}(k)=k$ there exists $i<k$ such that $w^{(k)}(i)=x$.  So 
\[([x,k-1],[k,\ell])=([w^{(k)}(i),k-1],[w^{(k)}(k),\ell])\in{\tt BoxStrands}.\]
 Conversely, assume 
 \[(I,J)=([w^{(k)}(i),k-1],[w^{(k)}(j),\ell])\in {\tt BoxStrands}\] 
and $I\cap J=\emptyset$.  Then $w^{(k)}(j)>k-1$ which means $w^{(k)}(j)=k$ and $j=k$ by the definition of $w^{(k)}$.  Furthermore, $w^{(k)}(i)\leq k-1$ since $i<j=k$.  So 
 \[(I,J)=([w^{(k)}(i),k-1],[k,\ell])\in {\tt ConditionStrands}\] is type (I).
\end{proof}

\begin{claim}
Let $(I,J)$ be as in (\ref{eqn:inthisform}).  Then $(I,J)\in {\tt BoxStrands} \Leftrightarrow (I,J)\in {\tt ConditionStrands}$.
\end{claim}

\begin{proof}

We will proceed by induction on $k\geq 2$. In the base case $k=2$, we must have $x=1$ and so $y\geq 2$.  As such, $I\cap J=\emptyset$ and so we are done Claim~\ref{claim:trivIntersect}.  
Fix $k> 2$ and assume the claim holds for $k-1$.  That is, given a pair of intervals $([x',k-2],[y',\ell'])$ so that $x',y'$ and $\ell'$ satisfy $x'\neq y'$ and $k-1\leq \ell'$ we have  
\begin{equation}
\label{eqn:inductionhyp}
([x',k-2],[y',\ell'])\in {\tt BoxStrands} \Leftrightarrow ([x',k-2],[y',\ell'])\in {\tt ConditionStrands}.
\end{equation}
Now let $(I,J)$ be as in (\ref{eqn:inthisform}), i.e., 
\[(I,J)=([x,k-1],[y,\ell]), \text{\ with $x\neq y$, $k\leq \ell$.}\]  
Again, by Claim~\ref{claim:trivIntersect}, if $I\cap J=\emptyset$ we are done, so assume $I\cap J\neq \emptyset$.  Then $y<k$. 

Now, since $1\leq x,y\leq k$, there exist $i$ and $j$ such that
\[1\leq i,j\leq k  \text{\ with $x=w^{(k)}(i)$ and $y=w^{(k)}(j)$.}\]  
So from (\ref{eqn:boxDef})
\begin{equation}
\label{eqn:boxiff}
(I,J)=([w^{(k)}(i),k-1],[w^{(k)}(j),\ell])\in {\tt BoxStrands}  \iff i<j.
\end{equation} Throughout, when $x\leq k-2$ we write $I':=[x,k-2]$.
We will break the argument into two main cases.  

\noindent Case 1: $a_{k-2}$ and $a_{k-1}$ point in the same direction.

By definition, $w^{(k)}=\iota(w^{(k-1)})$.  Then if $x\leq k-2$, it follows that
\begin{flalign*}
(I',J)&=([x,k-2],[y,\ell])\\
&=([w^{k-1}(i),k-2],[w^{k-1}(j),\ell])
\end{flalign*}
and so 
\begin{equation}
\label{eqn:I'J}
(I',J)\in {\tt BoxStrands} \text{ if and only if } i<j.
\end{equation}

We have four possible 
subcases, based on the relative values of $x$ and~$y$.

\noindent Subcase i: $x<y=k-1$.

$(I,J)$ is of type (II), and hence $(I,J)\in {\tt ConditionStrands}$.  Furthermore, note that  \[(I',J)=([x,k-2],[k-1,\ell])\] 
is of type (I), and so in {\tt ConditionStrands}.  The intervals for $(I',J)$ and $(I,J)$ look like this:
\begin{center}
\begin{tikzpicture}[x=.75cm,y=.75cm]
\filldraw [color=black,fill=black,thick](-1,1)circle(.1);
\filldraw [color=black,fill=black,thick](2,1)circle(.1);
\filldraw [color=black,fill=black,thick](3,0)circle(.1);
\filldraw [color=black,fill=black,thick](6,0)circle(.1);
\draw[dashed] (-1,1)--(2,1);
\draw[dashed] (3,0)--(6,0);
\node[] at (-1,1.5) {$x$};
\node[] at (2,1.5){$k-2$};
\node[] at (3,-.5) {$k-1$};
\node[] at (6,-.5) {$\ell$};
\filldraw [color=black,fill=black,thick](9,1)circle(.1);
\filldraw [color=black,fill=black,thick](12,1)circle(.1);
\filldraw [color=black,fill=black,thick](13,1)circle(.1);
\filldraw [color=black,fill=black,thick](13,0)circle(.1);
\filldraw [color=black,fill=black,thick](14,0)circle(.1);
\filldraw [color=black,fill=black,thick](16,0)circle(.1);
\draw[dashed] (9,1)--(13,1);
\draw[dashed] (13,0)--(16,0);
\node[] at (9,1.5) {$x$};
\node[] at (13,1.5){$k-1$};
\node[] at (13,-.5) {$k-1$};
\node[] at (16,-.5) {$\ell$};
\draw[ decoration={markings, mark=at position 0.6 with {\arrow[scale=2]{>}}}, postaction={decorate}](12,1) -- (13,1);
\draw[ decoration={markings, mark=at position 0.6 with {\arrow[scale=2]{>}}}, postaction={decorate}](13,0) -- (14,0);
\end{tikzpicture}.
\end{center}
 By the inductive hypothesis (\ref{eqn:inductionhyp}), $(I',J)\in {\tt BoxStrands}$. By (\ref{eqn:I'J}), $i<j$.  Therefore, by (\ref{eqn:boxiff}), $(I,J)\in{\tt BoxStrands}$.

Therefore, $(I,J)$ is in both ${\tt ConditionStrands}$ and ${\tt BoxStrands}$.

\noindent Subcase ii: $x<y<k-1$.

\begin{flalign*}
(I,J) \in {\tt BoxStrands} & \iff  i<j \quad \text{by (\ref{eqn:boxiff})}\\
&\iff (I',J)\in{\tt BoxStrands} \text{ by (\ref{eqn:I'J})}\\
& \iff (I',J) \in {\tt ConditionStrands} \text{ by (\ref{eqn:inductionhyp})}\\
& \iff a_{x-1} \text{ points in the same direction as } a_{k-2}\\
& \iff a_{x-1} \text{ points in the same direction as } a_{k-1}\\
& \iff (I,J)\in {\tt ConditionStrands}.
\end{flalign*}
The following picture depicts $(I',J)$ and $(I,J)$ respectively when $(I',J)$ and $(I,J)$ are in ${\tt ConditionStrands}$.
\begin{center}
\begin{tikzpicture}[x=.75cm,y=.75cm]
\filldraw [color=black,fill=black,thick](-1,1)circle(.1);
\filldraw [color=black,fill=black,thick](3,1)circle(.1);
\filldraw [color=black,fill=black,thick](2,0)circle(.1);
\filldraw [color=black,fill=black,thick](6,0)circle(.1);
\draw[dashed] (-1,1)--(3,1);
\draw[dashed] (2,0)--(6,0);
\node[] at (2,-.5) {$y$};
\node[] at (6,-.5) {$\ell$};
\node[] at (-1,1.5) {$x$};
\node[] at (3,1.5){$k-2$};
\filldraw [color=black,fill=black,thick](2,1)circle(.1);
\filldraw [color=black,fill=black,thick](1,1)circle(.1);
\filldraw [color=black,fill=black,thick](3,0)circle(.1);
\filldraw [color=black,fill=black,thick](4,0)circle(.1);
\draw[ decoration={markings, mark=at position 0.6 with {\arrow[scale=2]{>}}}, postaction={decorate}](1,1) -- (2,1);
\draw[ decoration={markings, mark=at position 0.6 with {\arrow[scale=2]{>}}}, postaction={decorate}](3,0) -- (4,0);

\filldraw [color=black,fill=black,thick](9,1)circle(.1);
\filldraw [color=black,fill=black,thick](13,1)circle(.1);
\filldraw [color=black,fill=black,thick](12,0)circle(.1);
\filldraw [color=black,fill=black,thick](16,0)circle(.1);
\draw[dashed] (9,1)--(13,1);
\draw[dashed] (12,0)--(16,0);
\node[] at (12,-.5) {$y$};
\node[] at (16,-.5) {$\ell$};
\node[] at (9,1.5) {$x$};
\node[] at (14,1.5){$k-1$};

\filldraw [color=black,fill=black,thick](14,1)circle(.1);
\filldraw [color=black,fill=black,thick](12,1)circle(.1);
\filldraw [color=black,fill=black,thick](11,1)circle(.1);
\filldraw [color=black,fill=black,thick](13,0)circle(.1);
\filldraw [color=black,fill=black,thick](14,0)circle(.1);
\filldraw [color=black,fill=black,thick](15,0)circle(.1);
\draw[ decoration={markings, mark=at position 0.6 with {\arrow[scale=2]{>}}}, postaction={decorate}](11,1) -- (12,1);
\draw[ decoration={markings, mark=at position 0.6 with {\arrow[scale=2]{>}}}, postaction={decorate}](13,0) -- (14,0);
\draw[ decoration={markings, mark=at position 0.6 with {\arrow[scale=2]{>}}}, postaction={decorate}](13,1) -- (14,1);
\draw[ decoration={markings, mark=at position 0.6 with {\arrow[scale=2]{>}}}, postaction={decorate}](14,0) -- (15,0);
\end{tikzpicture}
\end{center}

\noindent Subcase iii: $y<x=k-1$.

Pictured below are the intervals $I$ and $J$.
\begin{center}
\begin{tikzpicture}[x=.75cm,y=.75cm]
\filldraw [color=black,fill=black,thick](2,1)circle(.1);
\filldraw [color=black,fill=black,thick](-1,0)circle(.1);
\filldraw [color=black,fill=black,thick](2,0)circle(.1);
\filldraw [color=black,fill=black,thick](3,0)circle(.1);
\filldraw [color=black,fill=black,thick](1,0)circle(.1);
\filldraw [color=black,fill=black,thick](5,0)circle(.1);
\draw[dashed] (-1,0)--(5,0);

\node[] at (-1,-.5) {$y$};
\node[] at (5,-.5) {$\ell $};
\node[] at (2,1.5) {$x$};
\draw[ decoration={markings, mark=at position 0.6 with {\arrow[scale=2]{>}}}, postaction={decorate}](1,0) -- (2,0);
\draw[ decoration={markings, mark=at position 0.6 with {\arrow[scale=2]{>}}}, postaction={decorate}](2,0) -- (3,0);
\end{tikzpicture}
\end{center}
Since $y<x$ and this case assumes $a_{k-2}$ and $a_{k-1}$ point in the same direction, 
$(I,J)$ cannot be of type (III) and is not in {\tt ConditionStrands}.  Since 
\[w^{(k)}=\iota w^{(k-1)} \text{\ and $w^{(k-1)}(k-1)=k-1$,}\] 
it follows that $i=k-1$.  Since 
\[y=w^{(k)}(j)=w^{(k-1)}(j)<k-1,\] 
it follows that $i>j$, and so by (\ref{eqn:boxiff})
\[(I,J)\not \in {\tt BoxStrands}.\]
Therefore, $(I,J)$ is in neither ${\tt ConditionStrands}$ nor ${\tt BoxStrands}$.

\noindent Subcase iv: $y<x<k-1$.

\begin{flalign*}
(I,J) \in {\tt BoxStrands} & \iff  i<j \quad \text{by (\ref{eqn:boxiff})}\\
&\iff (I',J)\in{\tt BoxStrands} \text{ by (\ref{eqn:I'J})}\\
& \iff (I',J) \in {\tt ConditionStrands} \text{ by (\ref{eqn:inductionhyp})}\\
& \iff a_{x-1} \text{ points in the opposite direction as } a_{k-2}\\
& \iff a_{x-1} \text{ points in the opposite direction as } a_{k-1}\\
& \iff (I,J)\in {\tt ConditionStrands}.
\end{flalign*}
Below are $(I'J)$ and $(I,J)$ respectively, in the case $(I',J),(I,J)\in {\tt ConditionStrands}$.
\begin{center}
\begin{tikzpicture}[x=.75cm,y=.75cm]
\filldraw [color=black,fill=black,thick](1,1)circle(.1);
\filldraw [color=black,fill=black,thick](3,1)circle(.1);
\filldraw [color=black,fill=black,thick](-1,0)circle(.1);
\filldraw [color=black,fill=black,thick](1,0)circle(.1);
\filldraw [color=black,fill=black,thick](4,0)circle(.1);
\filldraw [color=black,fill=black,thick](3,0)circle(.1);
\filldraw [color=black,fill=black,thick](0,0)circle(.1);
\filldraw [color=black,fill=black,thick](6,0)circle(.1);
\draw[dashed] (-1,0)--(6,0);
\draw[dashed] (1,1)--(3,1);
\node[] at (-1,-.5) {$y$};
\node[] at (6,-.5) {$\ell$};
\node[] at (1,1.5) {$x$};
\node[] at (3.5,1.5){$k-2$};
\draw[ decoration={markings, mark=at position 0.6 with {\arrow[scale=2]{<}}}, postaction={decorate}](0,0) -- (1,0);
\draw[ decoration={markings, mark=at position 0.6 with {\arrow[scale=2]{>}}}, postaction={decorate}](3,0) -- (4,0);

\filldraw [color=black,fill=black,thick](11,1)circle(.1);
\filldraw [color=black,fill=black,thick](13,1)circle(.1);
\filldraw [color=black,fill=black,thick](14,1)circle(.1);
\filldraw [color=black,fill=black,thick](9,0)circle(.1);
\filldraw [color=black,fill=black,thick](11,0)circle(.1);
\filldraw [color=black,fill=black,thick](15,0)circle(.1);
\filldraw [color=black,fill=black,thick](14,0)circle(.1);
\filldraw [color=black,fill=black,thick](13,0)circle(.1);
\filldraw [color=black,fill=black,thick](10,0)circle(.1);
\filldraw [color=black,fill=black,thick](16,0)circle(.1);
\draw[dashed] (9,0)--(16,0);
\draw[dashed] (11,1)--(13,1);
\node[] at (9,-.5) {$y$};
\node[] at (16,-.5) {$\ell$};
\node[] at (11,1.5) {$x$};
\node[] at (14,1.5){$k-1$};
\draw[ decoration={markings, mark=at position 0.6 with {\arrow[scale=2]{<}}}, postaction={decorate}](10,0) -- (11,0);
\draw[ decoration={markings, mark=at position 0.6 with {\arrow[scale=2]{>}}}, postaction={decorate}](13,0) -- (14,0);
\draw[ decoration={markings, mark=at position 0.6 with {\arrow[scale=2]{>}}}, postaction={decorate}](14,0) -- (15,0);
\draw[ decoration={markings, mark=at position 0.6 with {\arrow[scale=2]{>}}}, postaction={decorate}](13,1) -- (14,1);
\end{tikzpicture}
\end{center}

\noindent Case 2: $a_{k-2}$ and $a_{k-1}$ point in opposite directions.

By definition, 
\[w^{(k)}=\iota( w^{(k-1)} w^{(k-1)}_0).\]    
 If $x\leq k-2$, and $y\leq k-1$ it follows that
\begin{flalign*}
(I',J)&=([x,k-2],[y,\ell])\\
&=([w^{(k-1)}(k-i),k-2],[w^{(k-1)}(k-j),\ell])
\end{flalign*}
and so 
\begin{equation}
\label{eqn:I'J2}
(I',J)\in {\tt BoxStrands} \text{ if and only if } k-i<k-j \text{ if and only if } i>j.
\end{equation}

\noindent Subcase i: $x<y=k-1$.

\begin{center}
\begin{tikzpicture}[x=.75cm,y=.75cm]
\filldraw [color=black,fill=black,thick](-1,1)circle(.1);
\filldraw [color=black,fill=black,thick](2,1)circle(.1);
\filldraw [color=black,fill=black,thick](3,0)circle(.1);
\filldraw [color=black,fill=black,thick](6,0)circle(.1);
\draw[dashed] (-1,1)--(2,1);
\draw[dashed] (3,0)--(6,0);
\node[] at (-1,1.5) {$x$};
\node[] at (2,1.5){$k-2$};
\node[] at (3,-.5) {$k-1$};
\node[] at (6,-.5) {$\ell$};
\filldraw [color=black,fill=black,thick](9,1)circle(.1);
\filldraw [color=black,fill=black,thick](12,1)circle(.1);
\filldraw [color=black,fill=black,thick](13,1)circle(.1);
\filldraw [color=black,fill=black,thick](13,0)circle(.1);
\filldraw [color=black,fill=black,thick](14,0)circle(.1);
\filldraw [color=black,fill=black,thick](16,0)circle(.1);
\draw[dashed] (9,1)--(13,1);
\draw[dashed] (13,0)--(16,0);
\node[] at (9,1.5) {$x$};
\node[] at (13,1.5){$k-1$};
\node[] at (13,-.5) {$k-1$};
\node[] at (16,-.5) {$\ell$};
\draw[ decoration={markings, mark=at position 0.6 with {\arrow[scale=2]{>}}}, postaction={decorate}](12,1) -- (13,1);
\draw[ decoration={markings, mark=at position 0.6 with {\arrow[scale=2]{<}}}, postaction={decorate}](13,0) -- (14,0);
\end{tikzpicture}
\end{center}
Since $a_{k-2}$ and $a_{k-1}$ point in opposite directions, $(I,J)\not \in {\tt ConditionStrands}$. The assumption $y=k-1$ implies $(I',J)\in {\tt ConditionStrands}$.   By  (\ref{eqn:inductionhyp}) $(I',J)\in{\tt BoxStrands}$. Since $x,y<k$, we have  \[x=w^{(k)}(i)=w^{(k-1)}(k-i) \text{\ and $y=w^{(k)}(j)=w^{(k-1)}(k-j)$.}\] 
Then $k-i<k-j$, so $i>j$ and $(I,J)\not \in {\tt BoxStrands}$, by (\ref{eqn:boxiff}).

Hence $(I,J)$ is neither in ${\tt ConditionStrands}$ nor ${\tt BoxStrands}$.

\noindent Subcase ii: $x<y<k-1$.
\begin{flalign*}
(I,J) \in {\tt BoxStrands} & \iff  i<j \text{ by (\ref{eqn:boxiff})}\\
&\iff (I',J)\not\in{\tt BoxStrands} \text{ by (\ref{eqn:I'J2})}\\
& \iff (I',J) \not \in {\tt ConditionStrands} \text{ by (\ref{eqn:inductionhyp})}\\
& \iff a_{y-1} \text{ points in the opposite direction as } a_{k-2}\\
& \iff a_{y-1} \text{ points in the same direction as } a_{k-1}\\
& \iff (I,J)\in {\tt ConditionStrands}.
\end{flalign*}
Below, we have $(I',J)\not\in {\tt ConditionStrands}$ and $(I,J)\in {\tt ConditionStrands}$.
\begin{center}
\begin{tikzpicture}[x=.75cm,y=.75cm]
\filldraw [color=black,fill=black,thick](-1,1)circle(.1);
\filldraw [color=black,fill=black,thick](3,1)circle(.1);
\filldraw [color=black,fill=black,thick](2,0)circle(.1);
\filldraw [color=black,fill=black,thick](6,0)circle(.1);
\draw[dashed] (-1,1)--(3,1);
\draw[dashed] (2,0)--(6,0);
\node[] at (2,-.5) {$y$};
\node[] at (6,-.5) {$\ell$};
\node[] at (-1,1.5) {$x$};
\node[] at (3,1.5){$k-2$};
\filldraw [color=black,fill=black,thick](2,1)circle(.1);
\filldraw [color=black,fill=black,thick](1,1)circle(.1);
\filldraw [color=black,fill=black,thick](3,0)circle(.1);
\filldraw [color=black,fill=black,thick](4,0)circle(.1);
\draw[ decoration={markings, mark=at position 0.6 with {\arrow[scale=2]{<}}}, postaction={decorate}](1,1) -- (2,1);
\draw[ decoration={markings, mark=at position 0.6 with {\arrow[scale=2]{>}}}, postaction={decorate}](3,0) -- (4,0);

\filldraw [color=black,fill=black,thick](9,1)circle(.1);
\filldraw [color=black,fill=black,thick](13,1)circle(.1);
\filldraw [color=black,fill=black,thick](12,0)circle(.1);
\filldraw [color=black,fill=black,thick](16,0)circle(.1);
\draw[dashed] (9,1)--(13,1);
\draw[dashed] (12,0)--(16,0);
\node[] at (12,-.5) {$y$};
\node[] at (16,-.5) {$\ell$};
\node[] at (9,1.5) {$x$};
\node[] at (14,1.5){$k-1$};

\filldraw [color=black,fill=black,thick](14,1)circle(.1);
\filldraw [color=black,fill=black,thick](12,1)circle(.1);
\filldraw [color=black,fill=black,thick](11,1)circle(.1);
\filldraw [color=black,fill=black,thick](13,0)circle(.1);
\filldraw [color=black,fill=black,thick](14,0)circle(.1);
\filldraw [color=black,fill=black,thick](15,0)circle(.1);
\draw[ decoration={markings, mark=at position 0.6 with {\arrow[scale=2]{<}}}, postaction={decorate}](11,1) -- (12,1);
\draw[ decoration={markings, mark=at position 0.6 with {\arrow[scale=2]{>}}}, postaction={decorate}](13,0) -- (14,0);
\draw[ decoration={markings, mark=at position 0.6 with {\arrow[scale=2]{>}}}, postaction={decorate}](13,1) -- (14,1);
\draw[ decoration={markings, mark=at position 0.6 with {\arrow[scale=2]{<}}}, postaction={decorate}](14,0) -- (15,0);
\end{tikzpicture}
\end{center}

\noindent Subcase iii: $y<x=k-1$. Here $(I,J)$ looks like:
\begin{center}
\begin{tikzpicture}[x=.75cm,y=.75cm]
\filldraw [color=black,fill=black,thick](2,1)circle(.1);
\filldraw [color=black,fill=black,thick](-1,0)circle(.1);
\filldraw [color=black,fill=black,thick](2,0)circle(.1);
\filldraw [color=black,fill=black,thick](3,0)circle(.1);
\filldraw [color=black,fill=black,thick](1,0)circle(.1);
\filldraw [color=black,fill=black,thick](5,0)circle(.1);
\draw[dashed] (-1,0)--(5,0);

\node[] at (-1,-.5) {$y$};
\node[] at (5,-.5) {$\ell $};
\node[] at (2,1.5) {$x$};
\draw[ decoration={markings, mark=at position 0.6 with {\arrow[scale=2]{>}}}, postaction={decorate}](1,0) -- (2,0);
\draw[ decoration={markings, mark=at position 0.6 with {\arrow[scale=2]{<}}}, postaction={decorate}](2,0) -- (3,0);
\end{tikzpicture}
\end{center}
Since Case~2 assumes $a_{k-2}$ and $a_{k-1}$ point in opposite directions, $(I,J)$ is type (II) and so in ${\tt ConditionStrands}$.  Now,
\[k-1=x=w^{(k)}(i)=w^{(k-1)}(k-i)\] 
which implies $i=1$.  Then $j>i$, so $(I,J)\in {\tt BoxStrands}$.
So $(I,J)$ is both in ${\tt ConditionStrands}$ and ${\tt BoxStrands}$.

\noindent Subcase iv: $y<x<k-1$.
\begin{align*}
(I,J) \in {\tt BoxStrands} & \iff  i<j \text{ by (\ref{eqn:boxiff})}\\
&\iff (I',J)\not\in{\tt BoxStrands} \text{ by (\ref{eqn:I'J2})}\\
& \iff (I',J) \not \in {\tt ConditionStrands} \text{ by (\ref{eqn:inductionhyp})}\\
& \iff a_{x-1} \text{ points in the same direction as } a_{k-2}\\
& \iff a_{x-1} \text{ points the opposite direction as } a_{k-1}\\
& \iff (I,J)\in {\tt ConditionStrands}. 
\end{align*} 
Pictured below are $(I',J)$ and $(I,J)$, in the case that $(I',J)\not\in {\tt ConditionStrands}$ and $(I,J)\in {\tt ConditionStrands}$.
\begin{center}
\begin{tikzpicture}[x=.75cm,y=.75cm]
\filldraw [color=black,fill=black,thick](1,1)circle(.1);
\filldraw [color=black,fill=black,thick](3,1)circle(.1);
\filldraw [color=black,fill=black,thick](-1,0)circle(.1);
\filldraw [color=black,fill=black,thick](1,0)circle(.1);
\filldraw [color=black,fill=black,thick](4,0)circle(.1);
\filldraw [color=black,fill=black,thick](3,0)circle(.1);
\filldraw [color=black,fill=black,thick](0,0)circle(.1);
\filldraw [color=black,fill=black,thick](6,0)circle(.1);
\draw[dashed] (-1,0)--(6,0);
\draw[dashed] (1,1)--(3,1);
\node[] at (-1,-.5) {$y$};
\node[] at (6,-.5) {$\ell$};
\node[] at (1,1.5) {$x$};
\node[] at (3.5,1.5){$k-2$};
\draw[ decoration={markings, mark=at position 0.6 with {\arrow[scale=2]{>}}}, postaction={decorate}](0,0) -- (1,0);
\draw[ decoration={markings, mark=at position 0.6 with {\arrow[scale=2]{>}}}, postaction={decorate}](3,0) -- (4,0);

\filldraw [color=black,fill=black,thick](11,1)circle(.1);
\filldraw [color=black,fill=black,thick](13,1)circle(.1);
\filldraw [color=black,fill=black,thick](14,1)circle(.1);
\filldraw [color=black,fill=black,thick](9,0)circle(.1);
\filldraw [color=black,fill=black,thick](11,0)circle(.1);
\filldraw [color=black,fill=black,thick](15,0)circle(.1);
\filldraw [color=black,fill=black,thick](14,0)circle(.1);
\filldraw [color=black,fill=black,thick](13,0)circle(.1);
\filldraw [color=black,fill=black,thick](10,0)circle(.1);
\filldraw [color=black,fill=black,thick](16,0)circle(.1);
\draw[dashed] (9,0)--(16,0);
\draw[dashed] (11,1)--(13,1);
\node[] at (9,-.5) {$y$};
\node[] at (16,-.5) {$\ell$};
\node[] at (11,1.5) {$x$};
\node[] at (14,1.5){$k-1$};
\draw[ decoration={markings, mark=at position 0.6 with {\arrow[scale=2]{>}}}, postaction={decorate}](10,0) -- (11,0);
\draw[ decoration={markings, mark=at position 0.6 with {\arrow[scale=2]{>}}}, postaction={decorate}](13,0) -- (14,0);
\draw[ decoration={markings, mark=at position 0.6 with {\arrow[scale=2]{<}}}, postaction={decorate}](14,0) -- (15,0);
\draw[ decoration={markings, mark=at position 0.6 with {\arrow[scale=2]{>}}}, postaction={decorate}](13,1) -- (14,1);
\end{tikzpicture}
\end{center}

\end{proof}

Theorem~\ref{prop:codim} now 
follows by combining Propositions~\ref{prop:first1}
and~\ref{prop:second1} with Lemma~\ref{lemma:bigfinale}.
\end{proof}

\section*{Acknowledgements}
We thank Bruce Berndt and Ae Ja Yee for historical references concerning Durfee square identities.  AW and AY were supported by a UIUC Campus Research Board and by an NSF grant.

\bibliographystyle{mybst}
\bibliography{mylib}

\begin{thebibliography}{ADF80}
\providecommand{\url}[1]{\texttt{#1}}
\providecommand{\urlprefix}{URL }
\expandafter\ifx\csname urlstyle\endcsname\relax
  \providecommand{\doi}[1]{doi:\discretionary{}{}{}#1}\else
  \providecommand{\doi}{doi:\discretionary{}{}{}\begingroup
  \urlstyle{rm}\Url}\fi

\bibitem[ADF80]{abeasis1980degenerations}
S.~Abeasis and A.~Del~Fra.
\newblock Degenerations for the representations of an equioriented quiver of
  type {Am}, boll.
\newblock \emph{Un. Mat. Ital. Suppl}, 2:157--171, 1980.

\bibitem[Bou02]{bouwknegt2002multipartitions}
P.~Bouwknegt.
\newblock Multipartitions, generalized {Durfee} squares and affine {Lie}
  algebra characters.
\newblock \emph{Journal of the Australian Mathematical Society},
  72(3):395--408, 2002.

\bibitem[Bri08]{brion2008representations}
M.~Brion.
\newblock Representations of quivers.
\newblock 2008.

\bibitem[DM16]{davison}
B.~Davison and S.~Meinhardt.
\newblock Cohomological {D}onaldson-{T}homas theory of a quiver with potential
  and quantum enveloping algebras.
\newblock \emph{arXiv:1601.02479}, 2016.

\bibitem[FG09]{fock}
V.~V. Fock and A.~B. Goncharov.
\newblock Cluster ensembles, quantization and the dilogarithm.
\newblock \emph{Ann. Sci. \'Ec. Norm. Sup\'er. (4)}, 42(6):865--930, 2009.

\bibitem[Kel11]{keller}
B.~Keller.
\newblock On cluster theory and quantum dilogarithm identities.
\newblock In \emph{Representations of algebras and related topics}, EMS Ser.
  Congr. Rep., pages 85--116. Eur. Math. Soc., Z\"urich, 2011.

\bibitem[KS11]{kontsevich}
M.~Kontsevich and Y.~Soibelman.
\newblock Cohomological {H}all algebra, exponential {H}odge structures and
  motivic {D}onaldson-{T}homas invariants.
\newblock \emph{Commun. Number Theory Phys.}, 5(2):231--352, 2011.

\bibitem[Rei01]{reineke2001feigin}
M.~Reineke.
\newblock Feigin's map and monomial bases for quantized enveloping algebras.
\newblock \emph{Mathematische Zeitschrift}, 237(3):639--667, 2001.

\bibitem[Rim13]{rimanyi2013cohomological}
R.~Rim{\'a}nyi.
\newblock On the cohomological {Hall} algebra of {Dynkin} quivers.
\newblock \emph{arXiv:1303.3399}, 2013.

\bibitem[Rin80]{ringel1980rational}
C.~M. Ringel.
\newblock The rational invariants of the tame quivers.
\newblock \emph{Inventiones mathematicae}, 58(3):217--239, 1980.

\bibitem[SW98]{schilling1998supernomial}
A.~Schilling and S.~O. Warnaar.
\newblock Supernomial coefficients, polynomial identities and q-series.
\newblock \emph{The Ramanujan Journal}, 2(4):459--494, 1998.

\end{thebibliography}

\end{document}